\documentclass{article}

\usepackage{graphicx}
\usepackage{amsmath}
\usepackage{amsfonts}
\usepackage{amssymb}

\usepackage[thref,amsmath]{ntheorem}

\makeatletter
\renewtheoremstyle{break}%
{\item[\rlap{\vbox{\hbox{\hskip\labelsep \theorem@headerfont
				##1\ ##2\theorem@separator}\hbox{\strut}}}]}%
{\item[\rlap{\vbox{\hbox{\hskip\labelsep \theorem@headerfont
				##1\ ##2:\ ##3\theorem@separator}\hbox{\strut}}}]}%
\makeatother

\theoremstyle{plain}

\usepackage{algorithm,algorithmic}
\usepackage{float}
\usepackage{epsfig}
\floatstyle{ruled}

\usepackage{adjustbox}
\usepackage{booktabs}
\usepackage{tabularx}
\usepackage{caption}
\usepackage{multirow}
\usepackage{siunitx} %
\usepackage{ragged2e}
\newcommand{\otoprule}{\midrule[\heavyrulewidth]} %

\usepackage{color}

\usepackage{easyReview}

\usepackage{footmisc}

\newcommand\blfootnote[1]{%
	\begingroup
	\renewcommand\thefootnote{}\footnote{#1}%
	\addtocounter{footnote}{-1}%
	\endgroup
}

\newtheorem{theorem}{Theorem}
\newtheorem{lemma}{Lemma}
\newtheorem{corollary}{Corollary}
\newtheorem{definition}{Definition}

\newtheorem{proposition}{Proposition}

\theoremstyle{nonumberplain}
\theoremheaderfont{\itshape}
\theorembodyfont{\normalfont}
\theoremseparator{.\,—}
\theoremsymbol{}
\newtheorem{proof}{Proof}

\theoremseparator{\,—}

\newcommand*{\email}[1]{%
	\normalsize #1 \par
}
\begin{document}

	\newcommand{\quod}{\hfill $\blacksquare$ \bigbreak}
	\newcommand{\reals}{I\!\!R}
	\newcommand{\np}{\mbox{{\sc NP}}}
	\newcommand{\sing}{\mbox{{\sc Sing}}}
	\newcommand{\con}{\mbox{{\sc Con}}}
	\newcommand{\prob}{\mbox{Prob}}
	\newcommand{\atm}{\mbox{{\sc ATM}}}
	\newcommand{\hopn}{\hop_{\cN}}
	\newcommand{\atmn}{\atm_{\cN}}
	\newcommand{\cF}{{\cal F}}
	\newcommand{\cM}{{\cal M}}
	\newcommand{\cN}{{\cal N}}
	\newcommand{\cP}{{\mathbb P}}
	\newcommand{\cD}{{\cal D}}
	\newcommand{\la}{{\langle}}
	\newcommand{\ra}{{\rangle}}
	\def\lalto{\left \lceil}
	\def\ralto{\right \rceil}
	\def\lbasso{\left \lfloor}
	\def\rbasso{\right \rfloor}
	\def\D{{\Delta}}
	\def\qed{\hfill$\Box$}

\title{On the Multi-Interval Ulam-R\'enyi Game: for $3$ lies $4$ intervals suffice}

\author{Ferdinando Cicalese and Massimiliano Rossi \\ \\ \small Department of Computer Science, University of Verona, Italy\\
	\email{ferdinando.cicalese@univr.it, massimiliano.rossi\_01@univr.it}}

\date{}

\maketitle		
		\begin{abstract}
			We study the problem of identifying an initially unknown $m$-bit number by using yes-no questions
			when up to a fixed number $e$ of the answers  can be erroneous.  
			In the variant we consider here questions are restricted to  
			be the union of up to a fixed number of intervals. 
			
			For any $e \geq 1$ let $k_e$ be the minimum $k$ such that 
			for all sufficiently large $m$,  
			there exists a strategy matching the information theoretic lower bound 
			and only using $k$-interval questions. It is known that $k_e = O(e^2)$ and 
			it has been conjectured that the $k_e = \Theta(e).$
			This  {\em linearity conjecture} is supported by the 
			known results for small values of $e$ as for $e\leq2$  
			we have $k_e = e.$ 
			We focus on the case $e=3$ and show $k_3 \leq 4$ 
			improving upon the previously known bound $k_3 \leq 10.$
			
		\end{abstract}

	\bibliographystyle{plain}
	\blfootnote{Part of the material of this paper was presented in preliminary form at ICTCS 2017 \cite{cicalesemulti}}
\section{Introduction}

In the {\bf Ulam-R\'enyi game with multi-interval questions,} two players, called Questioner and Responder---for reasons which 
will become immediately clear---fix three integer parameters: 
$m \geq 0,$ $e \geq 0$ and $k \geq 1.$ Then, Responder chooses a number $x$ from the set 
${\cal U} = \{0, 1, \dots, 2^m -1\},$ and keeps it secreto to Questioner.  
The task of Questioner is to identify the {\em secret} number $x$ chosen by Responder asking $k$-interval queries. 
These are 
$yes$-$no$ membership questions of the type ``Does $x$ belong to $Q$?'' where $Q$ 
is any subset of the search space 
{\em which can be expressed as the union of  $\leq k$ intervals}. We identify  a question with the subset $Q.$ Therefore, the
set of allowed questions is given by the family of sets: 
$${\cal T} = \left\{\bigcup_{i=1}^k \{a_{i}, a_{i}+1, \dots, b_{i}\} \mid 0\leq a_1 \leq b_1 \leq a_2 \leq b_2 \leq \cdots \leq a_k \leq b_k < 2^m\right\}.$$

Responder tries to  make Questioner's search as long as possible. To this aim, Responder can adversarially lie 
up to $e$ times during  game
i.e., by answering $yes$ to a question whose correct answer is $no$ or vice versa.

For any $m$ and $e$ let $N_{\min}(2^m, e) = \min\{q \mid 
2^{q-m} \geq \sum_{i=0}^e {q \choose i}\}.$ It is known that 
in a game over a search space of cardinality $n = 2^m$ and $e$ lies allowed to Responder, 
$N_{\min}(2^m, e)$ is a lower bound on the number of questions that  
Questioner has to ask in order to be sure to identify Responder's secret number 
 (see, e.g., \cite{Ber}).  
This lower bound holds in the  
version of the game in which questions  can refer to any subset, without the restriction 
to   $k$-interval queries. A fortiori, the lower bound holds for the multi-interval game for any value of $k$.
Strategies of size $N_{\min}(2^m,e),$ i.e., matching the information theoretic lower bound, are called {\em perfect}.
 
It is  known  that for any $e\geq 0$ and up to finitely many exceptional values of $m,$ Questioner can infallibly 
discover Responder's secret  number asking $N_{\min}(2^m,e)$ questions. 
However, in general such {\em perfect} strategies rely
 on the availability of arbitrary subset questions, i.e., any subset of the search space \cite{Spe}. 

When the cardinality of the search space is $2^m$ the description of an arbitrary subset query requires $2^m$ bits.
Moreover, in order to implement the known strategies using $N_{\min}(2^m, e)$ queries,   
$\Theta(e 2^m)$ bits are necessary to record the intermediate {\em states} of the game (see the next section for the details).
In contrast, a $k$-interval-query can be expressed by simply providing the boundaries of the $k$-intervals defining the 
question, hence reducing the space requirements of the strategy to only $2k \cdot m$ bits.

On the basis of these considerations, we will focus on the following problem:

\medskip

\noindent
{\bf Main question.} For given $e \geq 0$, denote by $k_e$ the smallest integer $k$ such that for all sufficiently large $m$ 
Questioner has a perfect strategy in the multi-interval game over the set of $m$-bit numbers only using 
$k$-interval queries. What is the value of $k_e$ for $e=1,2,\dots$ ?

\medskip

In \cite{mundici1997optimal} it was  proved that  for $e=2$ for any $m \geq 0$ (up to finitely many exceptions) there exists a searching 
strategy for Questioner of size $N_{\min}(m,e)$ (hence perfect) only using $2$-interval questions, and, conversely,  
perfect strategies which only use $1$-interval questions cannot generally exist. 
The case $e=1$ is analysed in \cite{Cicalese} where it is shown that perfect strategies exists for $e=1$ even when only 
using $1$-interval questions. However, simple comparison questions, namely
yes-no questions of the type ``Is $x \leq q$?'',  are not powerful enough to provide perfect strategies for the 
case $e=1$ \cite{Spe2,Aul}. 

These results show that for $e \leq 2,$ the answer to our main question is $k_e = e.$

In \cite{Cicalese} one of the authors proved that for any $e \geq 1$  there exists $k = O(e^2)$ such that for all sufficiently large $m$
Questioner can identify an $m$-bit number by using exactly $N(2^m,e)$ $k$-interval questions when 
Responder can lie at most $e$ times. 
In \cite{Cicalese} also conjectured that $k_e = O(e)$ interval might suffice for any $e$ and
all sufficiently large $m$. We will refer to this as the {\em linearity conjecture}.

\medskip

\noindent
{\bf Our Result.}
We focus on the case $e=3$. We first show that for any sufficiently large $m$, 
there exists a strategy  to identify an initially unknown 
$m$-bit number when up to $3$ answers are lies, which matches the information theoretic lower bound
and only uses $4$-interval queries. 
We then show how to refine our main tool to turn the above asymptotic result into a complete characterization of the 
instances of the Ulam-R\'enyi game with 4-interval question and $3$ lies that admit strategies using the theoretical minimum number of questions. 
For this, we build upon  the result of  \cite{negro1992ulam} and show that if there exists  a strategy for the {\em classical} Ulam-R\'enyi game with 
3 lies over a serach space that uses the information theoretic minimum number of questions, then the same strategy can be implemented 
using only $4$-interval questions. 

With reference to the {\em Main question} posed above, these results show  that $k_3 \leq 4,$ which significantly 
improves the best previously known bound of \cite{Cicalese} yielding 
$k_3 \leq  10.$ More interesting,  our novel analytic tool (Lemma \ref{lemma:well-shapeness}) naturally lends itself to a generalization 
to any fixed $e.$ The main open question is how to generalize  Theorem \ref{th:generalized-questions} in order to prove 
 the linearity conjecture $k_e = O(e).$ 

\medskip

\noindent
{\bf Related Work.}
The Ulam-R\'enyi game \cite{Ula,Ren2} has been extensively studied in various contexts including 
error correction codes \cite{ACDV,Ber,Cic-logic,deppe,Pel10}, 
learning \cite{Ces1,Learn,Or}, many-valued logics \cite{Mar,Cic-logic}, wireless networks \cite{Koz},  
psychophysics \cite{Kel-thesis,Kel}, and, principally, sorting and searching in the presence of errors 
(for the large literature on this topic, and the several variants studied, we refer the reader to the survey papers 
\cite{Mar,deppe,Pel0} and the book \cite{Cic-book}).

\section{Basic facts}

From now on we concentrate on the case $e = 3$ and the $4$-interval questions. 
Let $Q$  be the subset defining the question asked by Questioner.
Let $\overline{Q}$ be the complement of $Q,$ i.e., $\overline{Q} = \{0,1,\dots, 2^m-1\}\setminus Q.$

If Responder answers $yes$ to question $Q$, then we say that any number $y \in Q$ {\em satisfies} the answer
and any $y \in \overline{Q}$
{\em falsifies} the answer.

If Responder answers $no$ to question $Q$, then we say that any number $y \in \overline{Q}$ {\em satisfies} the answer
and any $y \in Q$
{\em falsifies} the answer.

At any stage of the game, we can partition the search space into $e+2$ subsets, 
$(A_0, A_1, A_2, A_{3}, A_{>3}),$ where $A_j$ ($j=0, \dots, 3$)
is the set of numbers falsifying exactly $j$ of Responder's answers, 
and therefore contains Responder's secret number if he has lied exactly $j$ times. 
Moreover, $A_{>3}$ is the set of numbers falsifying more than $e$ of the answers. 
Therefore, $A_{>3}$  is the set of numbers that cannot be the secret, because, otherwise, 
Responder would have lied too many times.
We refer to the vector $(A_0, A_1, A_2, A_3, A_{>3})$ as the {\em state} of the game, since it 
is a complete record of Questioner's state of knowledge 
on the numbers which are candidates to be Responder's 
secret number, 
as resulting from the sequence of questions/answers exchanged so far.

We will find it comfortable to have an alternative perspective on the state of the game. 
For a state $\sigma = (A_0, A_1, A_2, A_3, A_{>3})$ and for each $y \in U$
we define $\sigma(y)$ as the number of answers falsified by $y$, truncated at $4.$ 
Then for any $i = 0, \dots, 3,$ we have $\sigma(y) = i$ iff  $y \in A_i.$ 
For each $i = 0, 1, \dots, 3,$ we will also write 
$\sigma^{-1}(i)$ to denote the set $A_i.$ And define $\sigma^{-1}(4) = A_{>3} = \overline{\cup_{i=0}^3 A_i}$
to denote the set of numbers that as a result of the answers received cannot be Responder's 
secret number. From now on, we refer to a state of the game as the state $\sigma = (A_0,...,A_e)$, since the set $A_{>e}$ can be reconstructed from the sets $A_0, \ldots , A_e$ and the universe ${\cal U}$.

\begin{definition}[States and Supports]
A {\em state} is a map $\sigma : {\cal U} \rightarrow \{0,1, 2, 3, 4\}$.
The {\em type} of $\sigma$ is the quadruple $\tau(\sigma) = (t_0, t_1, t_2, t_3)$ where 
$t_i = |\sigma^{-1}(i)|$ for each $i=0,1,2,3.$ 
The {\em support} $\Sigma$ of $\sigma$ is the set of all $y \in {\cal U}$ such that $\sigma (y) < 4.$
A state is {\em final} iff its support has cardinality at most one.
The {\em initial state} $\alpha$ is the function mapping each element of ${\cal U}$ to $0.$
\end{definition}

Let $(A_0, \dots, A_e)$ be the current state and $Q$ be the new question asked by Questioner.
Questioner's new 
state of knowledge if Responder answers $yes$ to $Q$ is obtained by the following rules:
$$A_0 \leftarrow A_0 \cap Q \quad \mbox{and} \quad 
A_j \leftarrow (A_j \cap Q) \cup (A_{j-1} \cap \overline{Q}), \, \mbox{for } j = 1, \dots, e
$$

If Responder answers $no$, the above definitions and rules apply with $Q$ replaced by its complement 
$\overline{Q} = \{0,1,\dots, 2^m-1\}\setminus Q,$
i.e., answering $no$ is the same as answering $yes$ to the complementary question $\overline{Q}$.

With the above functional notation we formalise this as follows:

\begin{definition}[Answers and Resulting States]
Let $\sigma$ be the current state  with support $\Sigma$ and $Q$ be the new question asked by Questioner. Let $b \in \{yes, no\}$ be 
the answer of Responder. 
Define the answer function $b: \Sigma \rightarrow \{0, 1\}$ associated to question $Q$ 
by stipulating that $b(y) = 0$ if and only if $y$ satisfies answer $b$ to question $Q$. 
Then, the resulting state $\sigma_b$ is given by 
$\sigma_b(y) = \min\{\sigma(y) + b(y), 4\}.$

More generally, starting from state $\sigma$ after questions $Q_1, \dots, Q_t$ with answers $b_1,\dots,b_t$ the 
resulting state is $\displaystyle{\sigma_{b_1\,b_2\, \cdots\, b_t} = \min\{\sigma(y) + \sum_{j=1}^t b_j(y), 4\}.}$
\end{definition} 

In particular, for $\sigma$ being the initial state we have that the resulting state after $t$ questions is the truncated sum of the 
corresponding answer functions associated to Responder's answers.

\begin{definition}[Strategy]
A strategy of size $q$ is a complete binary tree of depth $q$ where each internal node $\nu$ maps to a question $Q_{\nu}.$ The left and right
branch stemming out of $\nu$ map to the  function answers $yes$ and $no$ associated to question $Q_{\nu}.$ 
Each leaf $\ell$ is associated to the state $\sigma^{\ell}$ resulting 
from the sequence of questions and answers  associated to the nodes and branches on unique path from the root to $\ell$. In particular, if 
$b_1, \dots, b_q$ are the answers/branches leading to $\ell$ then we have $\sigma^{\ell} = \alpha_{b_1\, \cdots\, b_q}$ as defined above.

The strategy is {\em winning} iff for all leaves $\ell$ we have that $\sigma^{\ell}$ is a final state.
\end{definition}

We can also extend the above definition to an arbitrary starting state. Given a state $\sigma$, we say that a strategy ${\cal S}$ of size $q$ is winning for $\sigma$ 
if for any root to leaf path in ${\cal S}$ with associated answers $b_1, \dots, b_q$ the state $\sigma_{b_1\, \cdots\, b_q}$ is final .

We define the {\em character} of a state $\sigma$ 
as  $ch(\sigma)= \min\{q \mid w_q(\sigma) \leq 2^q\},$ where 
$w_q(\sigma)  = \sum_{j=0}^3 |\sigma^{-1}(j)| \sum_{\ell = 0}^{3-j} {q \choose \ell}$
is referred to as the {\em $q$th volume} of  $\sigma.$

We have that the lower bound  $N_{\min}(2^m,e),$ mentioned in the introduction, coincides with the 
character of the initial state $\sigma^{0} = \alpha =  ({\cal U}, 0, \dots, 0)$ 
(see  Proposition \ref{proposition:volume_properties} below). 
Notice also that a state  has character $0$ if and only if it is a final state.

For a state $\sigma$ and a question $Q$ let 
$\sigma_{yes}$ and $\sigma_{no}$ be the resulting states according to whether 
Responder answers, respectively, yes or no, to question $Q$
in state $\sigma$.
Then, from the definition of the $q$th volume of a state, it follows that
for each $q \geq 1,$ we have $w_{q}(\sigma) = w_{q-1}(\sigma_{yes}) + w_{q-1}(\sigma_{no}).$
A simple induction argument gives the following lower bound \cite{Ber}. 
 
\begin{proposition} \label{proposition:volume_properties}Let $\sigma$ be the state of the game.  
For any integers $0 \leq q< ch(\sigma)$ and $k\geq 1,$ starting from 
state $\sigma,$ Questioner cannot determine Responder's secret number asking only $q$ 
many $k$-interval-queries.
\end{proposition}

In order to finish the search within $N_{\min}(2^m, e)$ queries, Questioner has to
guarantee that each question asked induces a strict decrease of the character
of the state of the game. The following 
lemma provides a sufficient condition for obtaining such a strict decrease of the character.

\begin{lemma} \label{lemma:balance}
Let $\sigma$ be the current state, with $q = ch(\sigma).$ Let $D$ be Questioner's question and 
$\sigma_{yes}$ and $\sigma_{no}$ be the resulting states according to whether 
Responder answers, respectively, yes or no, to question $D.$

If $|w_{q-1}(\sigma_{yes}) - w_{q-1}(\sigma_{no})| \leq 1$ then it holds that 
$ch(\sigma_{yes}) \leq q-1$ and $ch(\sigma_{no}) \leq q-1.$
\end{lemma}
\begin{proof}
Assume, w.l.o.g., that $w_{q-1}(\sigma_{yes}) \geq w_{q-1}(\sigma_{no}).$ 
Then, from the hypothesis, it follows that $w_{q-1}(\sigma_{no}) \geq w_{q-1}(\sigma_{yes})-1.$
By definition of character we have 
$\displaystyle{2^q \geq w_q(\sigma) = w_{q-1}(\sigma_{yes}) + w_{q-1}(\sigma_{no}) \geq 2w_{q-1}(\sigma_{yes}) - 1,}$
hence, $w_{q-1}(\sigma_{no}) \leq w_{q-1}(\sigma_{yes}) \leq 2^{q-1}+1/2,$ which 
together with the integrality of the volume, implies that for both $\sigma_{yes}$ and $\sigma_{no}$
the $(q-1)$th volume is not larger than $2^{q-1},$ hence their character is not larger than $q-1,$ as
desired.
\qed
\end{proof}

A question which satisfies the hypothesis of Lemma \ref{lemma:balance} will be called {\em balanced}.
A special case of balanced question is obtained when for a state $\sigma = (A_0, \dots, A_e)$ the question 
$D$ is such that $|D \cap A_i| = |A_i|/2$ for each $i=0, \dots, e.$ In this case, we also call the question 
an {\em even splitting}.

\begin{definition}[Interval]
An interval in $\mathcal{U}$ is either the empty set $\emptyset$ or a set of consecutive elements  $[a,b] = \{x\in\mathcal{U} | a \leq x \leq b\}.$ 
The elements $a,b$ are called the boundaries of the interval.
\end{definition}

\begin{definition}[4-interval-question]
A $4$-interval question (or simply a question) is any subset $Q$ of $\mathcal{U}$ 
such that  $Q = I_1 \cup I_2 \cup I_3 \cup I_4$ where for $j = 1,2,3,4,$ $I_j$ is an interval 
in $\mathcal{U}$.

The {\em type of $Q$}, denoted by $|Q|$, is the quadruple $|Q| = [a_0^{Q},a_1^{Q},a_2^{Q},a_3^{Q}]$ where
for each $i = 0,1,2,3$ , $a_i^{Q} = |Q \cap \sigma^{-1}(i)|.$
\end{definition}

Following \cite{mundici1997optimal} we visualize the search space as a necklace and restrict it to the set of
number which are candidate to be the secret number, i.e., we identify $\cal U$ with its support $\Sigma = {\cal U} \cap \bigcup_{i=0}^3 A_i.$

For any non-final state, i.e., $|\bigcup_{i=0}^3 A_i| > 1$, for each $x \in  \bigcup_{i=0}^3 A_i$  
we define the successor of $x$ to be  the number $x+r \mod 2^m$ for 
the smallest $0 < r < 2^m$ such that $x+ r \mod 2^m \in \bigcup_{i=0}^3 A_i.$ In particular, for the initial state, $0$ is the 
successor of $2^m-1.$

For $a, b \in \bigcup_{i=0}^3 A_i$ we say that there is an arc from $a$ to $b$ and denote it by $\langle a , b \rangle$
if the following two conditions hold: (i)  $\sigma(x) = \sigma(a)$  for each element $x$ encountered when moving from 
 $a$ to $b$ in ${\cal U}$ passing from one element to its successor; (ii) $\sigma(c) \neq \sigma(a)$ and 
$\sigma(d) \neq \sigma(b)$ where $a$ is the successor of $c$ and $d$ is the successor of $b$. 
We say that arc $\langle a , b \rangle$ is on level $\sigma(a)$ and call $a$ and $b$ the left and right boundary of the arc.

In words, an arc is a maximal sequence of consecutive elements lying on the same level of the state $\sigma$.

For the sake of definiteness, we allow an arc to be empty. Therefore, we can associate to a state $\sigma$ 
a smallest sequence ${\cal L}^{\sigma}$ of
(possibly empty) arcs $a_0, \dots, a_{r-1}$ such that for each $i$ the levels of arcs $a_i$ and $a_{i+1}$ differ exactly by $1$. Note that, 
by requiring that the length $r$ be minimum the sequence ${\cal L}^{\sigma}$ is uniquely determined up to circular permutation. 

For each $i = 0, 1, \dots, r,$ we say that arcs $a_i$ and $a_{(i+1) \bmod r}$ are {\em adjacent} (or {\em neighbours}). 
We say that $a_i$ is a {\em saddle} if both adjacent arcs are 
on a lower level, i.e., $a_{(i-1) \bmod r}, a_{(i+1) \bmod r} \in \sigma^{-1}(k-1)$ and $a_i \in \sigma^{-1}(k)$ for some $k$

We say that $a_i$ is a {\em mode} if both adjacent arcs are 
on a higher level, i.e., $a_{(i-1) \bmod r}, a_{(i+1) \bmod r} \in \sigma^{-1}(k+1)$ and $a_i \in \sigma^{-1}(k)$ for some $k$

We say that $a_i$ is a {\em step} if for some $k$ $a_i \in \sigma^{-1}(k)$ and either $a_{(i-1) \bmod r} \in \sigma^{-1}(k-1), a_{(i+1) \bmod r} \in \sigma^{-1}(k+1)$ or $a_{(i-1) \bmod r} \in \sigma^{-1}(k+1), a_{(i+1) \bmod r} \in \sigma^{-1}(k-1)$.

Based on the above notions, we now define a well-shaped state for $e$ lies.

\begin{definition}\label{def:well-shaped}%
Let $\sigma$ be a state and $\mathcal{L}^{\sigma}$ be its associated list of arcs. Then, 
$\sigma$ is \textbf{well shaped} iff the following conditions hold:
\begin{itemize}
	\item for  $i=0, \dots, e-1$, in ${\cal L}^{\sigma}$ there are exactly $(2i+1)$ arcs lying on level $i$. 
	\item  in ${\cal L}^{\sigma}$ there are exactly $e$ arcs lying on level $e$.
\end{itemize}

\end{definition}

It is not hard to see that for the case $e=3$ under investigation, the only two feasible well-shaped states 
are described as follow:  $\sigma^{-1}(0)$ is an arc $S$ in $\Sigma$; 
$\sigma^{-1}(1)$ is the disjoint union of three arcs $H$,$N$ and $O$ in $\Sigma$ with $N$ and $O$ adjacent to $S$;
$\sigma^{-1}(2)$ is the disjoint union of five arcs  $A,B,C,L$ and $M$ in $\Sigma$ with $B$ and $C$ adjacent to $H$, $L$ adjacent to $N$; $\sigma^{-1}(3)$ is the disjoint union of three arcs $P,Q,R$ in $\Sigma$ with $R$ and $P$ adjacent to $A$. 

Starting with $S$ and scanning $\mathcal{U}$ with positive orientation, we can list the twelve arcs (restricted to $\Sigma$) in one of the following two possibilities:
\begin{eqnarray} \label{eq:well shaped}
 \sigma_1 &=& L^{2} N^{1} S^{0} O^{1} M^{2} Q^{3} B^{2} H^{1} C^{2} R^{3} A^{2} P^{3} \label{eq:well shaped}\\
 \sigma_2 &=& L^{2} N^{1} S^{0} O^{1} B^{2} H^{1} C^{2} Q^{3} M^{2} R^{3} A^{2} P^{3} \label{eq:well shaped2}
\end{eqnarray}

where for an arc $X$ the notation $X^i$ is meant to denote the fact that $X \subseteq \sigma^{-1}(i).$

Typical well-shaped states of type (\ref{eq:well shaped})  and (\ref{eq:well shaped2}) are shown in Figures \ref{fig.well-shaped} and \ref{fig.well-shaped2} respectively.

\begin{figure*}[h]
	\centering \includegraphics[width=\textwidth]{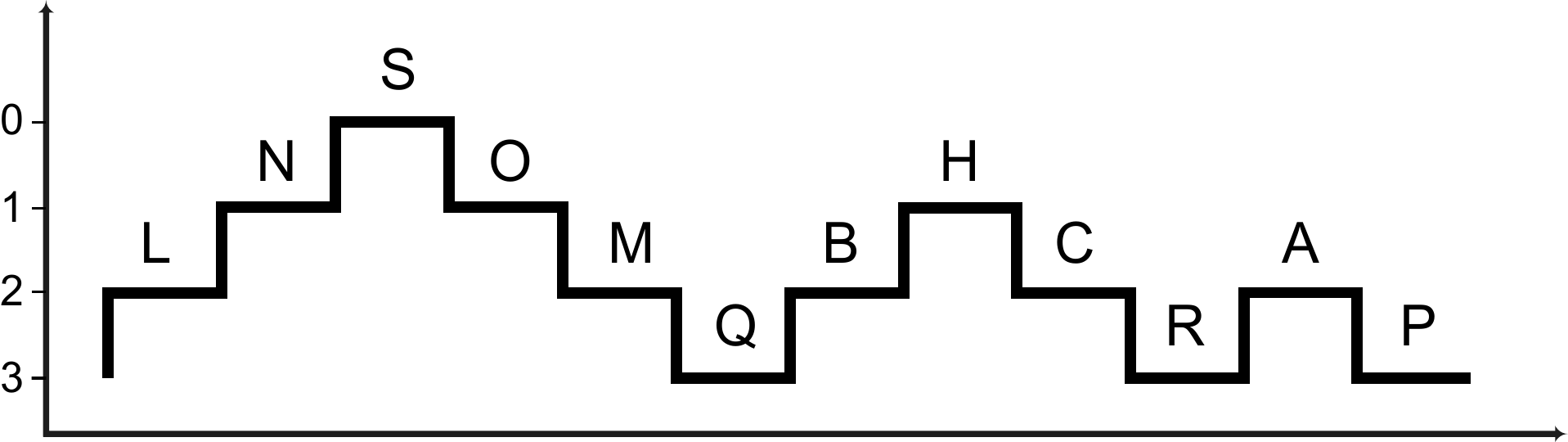}
	\caption{A well-shaped state of type (\ref{eq:well shaped}). Arcs $S, H, A$ are {\em modes} at level $0,1,2,$ respectively.
		Arcs $Q, R, P$ are {\em saddles} at level $3$. The remaining arcs are {\em steps}. The arrow shows the positive orientation.}
	\label{fig.well-shaped}
\end{figure*}

\begin{figure*}[h]
	\centering \includegraphics[width=\textwidth]{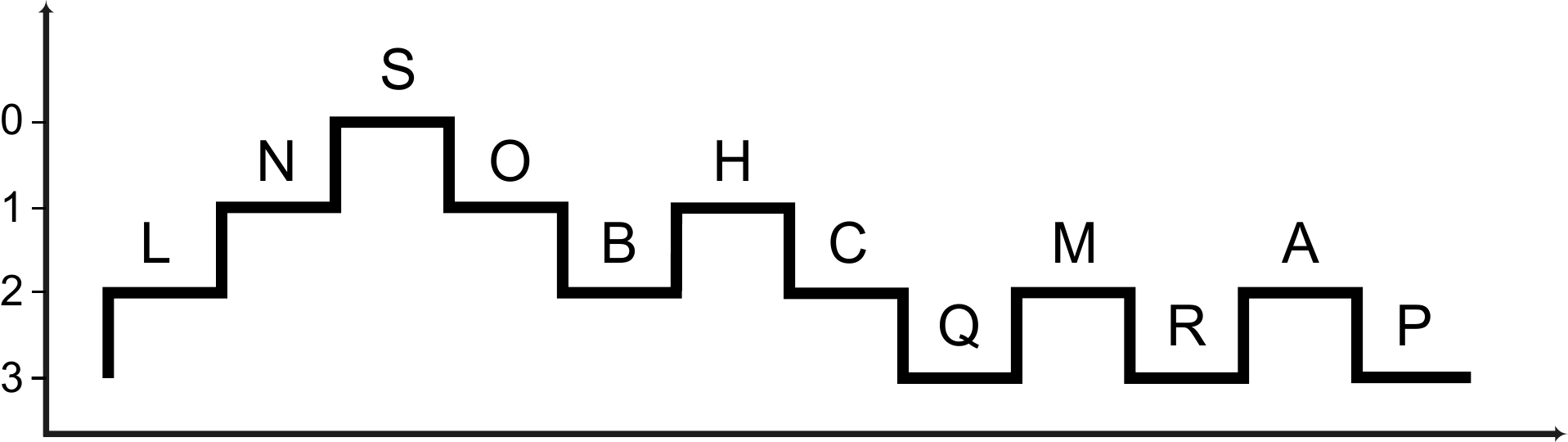}
	\caption{A well-shaped state of type (\ref{eq:well shaped2}). Arcs $S, H, M, A$ are {\em modes}.
		Arcs $B, Q, R, P$ are {\em saddles}. The remaining arcs are {\em steps}. The arrow shows the positive orientation.}
	\label{fig.well-shaped2}
\end{figure*}

\section{The key result - structure of perfect 4-interval strategies for 3 lies}

The strategy we propose is based on the approach of Spencer \cite{Spe}. 
We are going to recall the strategy presented in  \cite{Spe} and prove how to 
implement questions in that strategy by means of $4$-interval questions. The main technical tool will be to 
show that we can define $4$-interval balanced questions and also guarantee that each intermediate state is well-shaped.

In this section, we will characterise questions in terms of the ratio between the components of the question and the components of 
the state they are applied. We will show conditions for the existence of questions that  
can be implemented using only 4-intervals and such that the resulting states are {\em well-shaped} whenever the state they are applied to is well shaped. 

For each of them, we show how to select the exact amount of elements in the query among the arcs representing the state. Then, we prove that no other cases are allowed and finally we show that the well shapeness property of the state is preserved in both states resulting from the answer to the query.

Our main technical tool is the following theorem, whose proof is deferred to the next section.

\begin{theorem} \label{th:generalized-questions}
Let $\sigma$ be a well-shaped state of type $\tau (\sigma) = (a_0, b_0,c_0,d_0) $.
For all integers
$\displaystyle{0 \leq a \leq a_0 \quad  0 \leq  b \leq \lceil \frac{1}{2} b_0 \rceil 
\quad 0 \leq  c \leq \lceil \frac{1}{2} c_0 \rceil  
\quad 0 \leq  d \leq \lceil \frac{2}{3} d_0 \rceil}
$
there exists a $4$-interval question $Q$ of type $|Q| = [a,b,c,d]$  that we can ask in state $\sigma$ and such 
that both the resulting "yes" and "no" states are \emph{well-shaped}.
\end{theorem}

Before going into the details of the proof, we now show how to employ Theorem  \ref{th:generalized-questions} 
to show that asymptotically, every instance of the Ulam-R\'enyi game with $3$ lies over a space of cardinality $2^m$
admits a perfect startegy that only uses $4$-interval questions. 
For this, we use the main result of  \cite{Spe} that rephrase in our setting $e=3$ guarantees that for all sufficiently large $m$,
the strategy using the following two steps is perfect for the Ulam-R\'enyi game with $3$ lies over the search space of size $2^m$: 

\begin{enumerate}
\item  As long as the state satisfies $\sum_{i=0}^2 |\sigma^{-1}(i)| > 1$ ask a question $Q$ of 
type $[a_0^Q, a_1^Q, a_2^Q, a_3^Q]$ where, for $i = 0,1,2,$ 
$a_i^Q \in \left\{\lfloor \frac{|\sigma^{-1}(i)|}{2} \rfloor, \lceil \frac{|\sigma^{-1}(i)|}{2} \rceil \right\}$ with the choice of whether to 
choose floor or ceiling alternating among those levels where $|\sigma^{-1}(i)|$ is odd. The value of $a_3^Q$ is appropriately computed, 
based on the choices of $a_0^Q, a_1^Q, a_2^Q$, in order to guarantee that the resulting question is balanced, i.e., 
$a_3^Q = \lfloor \frac{1}{2} (\sum_{j=0}^{2} (|\sigma^{-1}(j)|-2a^Q_j){{q} \choose {j}} + \sigma^{-1}(3) )\rfloor$, where $q+1$ is the 
character of $\sigma$;
\item when the state satisfies  $\sum_{i=0}^2 |\sigma^{-1}(i)| \leq 1$, ask a balanced question $Q$ of type $[0,0,0, a_3^Q]$. \end{enumerate}

The main point in the argument of \cite{Spe} is that, up to finitely many exceptions, for all $m = \log |{\cal U}|$, the value 
$a_3^Q$ defined in 1.\ is feasible, in the sense that using the above rules yields $0 \leq a_3^Q \leq |\sigma^{(-1)}(3)|$.

We can now employ Theorem \ref{th:generalized-questions} to show that the above step can be 
implemented by a 4-intervals-question.
Let $Q$ be the question defined in $1.$ Let $\overline{Q}$ be the complementary question, and 
$[a_0^{\overline{Q}},  a_1^{\overline{Q}}, a_2^{\overline{Q}}, a_3^{\overline{Q}}]$ denote its type.
For $i=0,1,2,$ we have $a_i^{Q}, a_i^{\overline{Q}} \leq \lceil \frac{|\sigma^{-1}(i)|}{2} \rceil.$ Moreover,
we have $\min \{ a_3^{Q}, a_3^{\overline{Q}} \} \leq \lceil \frac{2}{3} |\sigma^{-1}(3)| \rceil.$
Therefore, asking $Q$ or  $\overline{Q}$ according to whether $a_3^{Q} \leq  a_3^{\overline{Q}}$
guarantees that the question satisfies the hypothesis of Theorem \ref{th:generalized-questions} 
and then it can be implemented as $4$-interval question which also preserves the well-shape of the state. 

\smallskip

The condition in 2. can also be easily guaranteed only relying on $4$-interval questions. 
In fact, the following proposition 
 shows that questions in point 2.\ are implementable by $4$-interval questions, preserving the 
well-shape of the state.

\begin{proposition} \label{proposition:end_game}
Let $\sigma$ be a well-shaped state with $\sigma^{-1}(3)>0$ and $\sum_{i=0}^{2} |\sigma^{-1}(i)| \leq 1.$
Let $ch(\sigma) = q.$ Then, starting in state $\sigma$ the Questioner can discover the Respounder's secret number asking exactly $q$ many $1$-interval-queries. 
\end{proposition}

\begin{proof}
We prove the proposition by induction on $q,$ the character of the state. 
If $q=1,$ the only possibility is $\sum_{i=0}^{2} \sigma^{-1}(i) = 0$ and $\sigma^{-1}(3) = 2.$ Then, 
a question containing exactly one of the elements in $\sigma^{-1}(3)$ is enough to conclude the
search.

Now assume that $q > 1$ and that the statement holds for any state with the same structure and character $\leq q-1.$

If $\sum_{i=0}^{2} \sigma^{-1}(i) = 0$ then the solution is provided by the classical binary search in the set $\sigma (e),$ which 
can be clearly implemented using $1$-interval-queries. Notice also that the number of intervals needed to 
represent the new state, is never more than the number of intervals needed to represent $\sigma.$

Assume now that $\sum_{i=0}^{2} \sigma^{-1}(i) \neq 0$ and let $j$ be the index such that $\sigma^{-1}(j) = 1.$ 
Let $c$ be the only element in the $j$th level.

Let $\alpha = \sum_{i=0}^{3-j} {{q-1} \choose i} \leq 2^{q-1}.$ By the assumption of the character of $\sigma$ follows that $\sum_{i=0}^{3-j} {q \choose i} \leq w_q(\sigma) \leq 2^{q},$ hence $3-j < q$ and in addition $\alpha \leq 2^{q-1}.$ 

In addition, we have $\sum_{i=0}^{3-j} {{q-1} \choose i} + \sigma^{-1}(3) >  2^{q-1}.$ 

Choose $I$ to be the largest interval including $c$ and $2^{q-1} - \alpha$ 
other elements from the $3$rd level. The above observation guarantees the existence of such interval.
The possible states arising from such a question satisfy $w_{q-1}(\sigma_{yes}) = 2^{q-1}$ and 
are $w_{q-1}(\sigma_{no}) = w_{q}(\sigma) - w_{q-1}(\sigma_{yes}) \leq 2^{q} - 2^{q-1}.$ 
Hence, both states have character not larger than $q-1.$ It is also not hard to see that 
they both have a structure satisfying the hypothesis of the 
proposition. This proves the induction step.

Note that the number of intervals necessary to represent the new state is not 
larger than the initial one, then the state is also well-shaped. \qed
\end{proof}

We have shown that any question in the perfect strategy of \cite{Spe} can be implemented by 
4-interval questions. We can summarise our discussion in the following theorem.

\begin{theorem}
For all sufficiently large $m$  in the game 
played over the search space $\{0, \dots, 2^m - 1\}$ with 
$3$ lies, 
there exists a \emph{perfect} 4-intervals strategy. In particular, 
the strategy uses at most $N_{\min}(2^m,3)$ questions and all the states of the game 
are well shaped, hence representable by exactly $12 \log m$ bits (12 numbers from $\cal U$). 
\end{theorem}

\section{The proof of Theorem \ref{th:generalized-questions}}
Let  $\sigma$ be a state and $\langle a, b \rangle$ be a non-empty arc of $\sigma.$ 
We say that a question $Q$ splits the 
arc  $\langle a, b \rangle$ if there exists an interval $I$ in $Q$ that intersects $\langle a, b \rangle$ and contains exactly one of its
 boundaries $a, b$. In words, there is an interval in the question  such that   
some non-empty part of the arc satisfies a yes answer and some non-empty part of the arc satisfies a no answer.

If a question $Q$ splits exactly one arc on level $i$ of $\sigma$ according to whether such an arc is a mode, a saddle, or a step, 
we say that {\em at level $i$} the question $Q$ (or, equivalently, an interval of $Q$)
is {\em mode-splitting}, {\em saddle-splitting}, {\em step-splitting}, respectively. 

Let  $Q$ be a {\em step-splitting} question at level $i$. Let $\langle a, b \rangle$ be the arc at level $i$ which is split by 
an interval $I$ of $Q$. Then, by definition $I$ contains exactly one of the boundaries of the arc. 
 If $I$ contains the boundary of the arc that flanks an arc at level $i+1$ we say that $Q$ 
 (or, equivalently, an interval of $Q$)
 is {\em downward step-splitting}; if $I$ contains the boundary of the arc that flanks an arc at level $i-1$ we say that 
 $Q$ (or, equivalently, an interval of $Q$) is {\em upward step-splitting}.

We say that a question $Q$ {\em covers entirely} the 
arc  $\langle a, b \rangle$ if $[a, b]$ is contained in one of the intervals defining $Q$. 

If a question $Q$ covers entirely  an arc on level $i$ of $\sigma$ according to whether such an arc is a mode, a saddle, a step, 
we say that {\em at level $i$} the question $Q$ (or, equivalently, an interval of $Q$)
 is {\em mode-covering}, {\em saddle-covering}, {\em step covering}, respectively. 
Refer to Figure \ref{fig.well-shaped-dynamic} for a pictorial representation of the interval types and questions effect on states.
\setlength{\belowcaptionskip}{-10pt}
\begin{figure}[h]
	\centering \includegraphics[width=\textwidth]{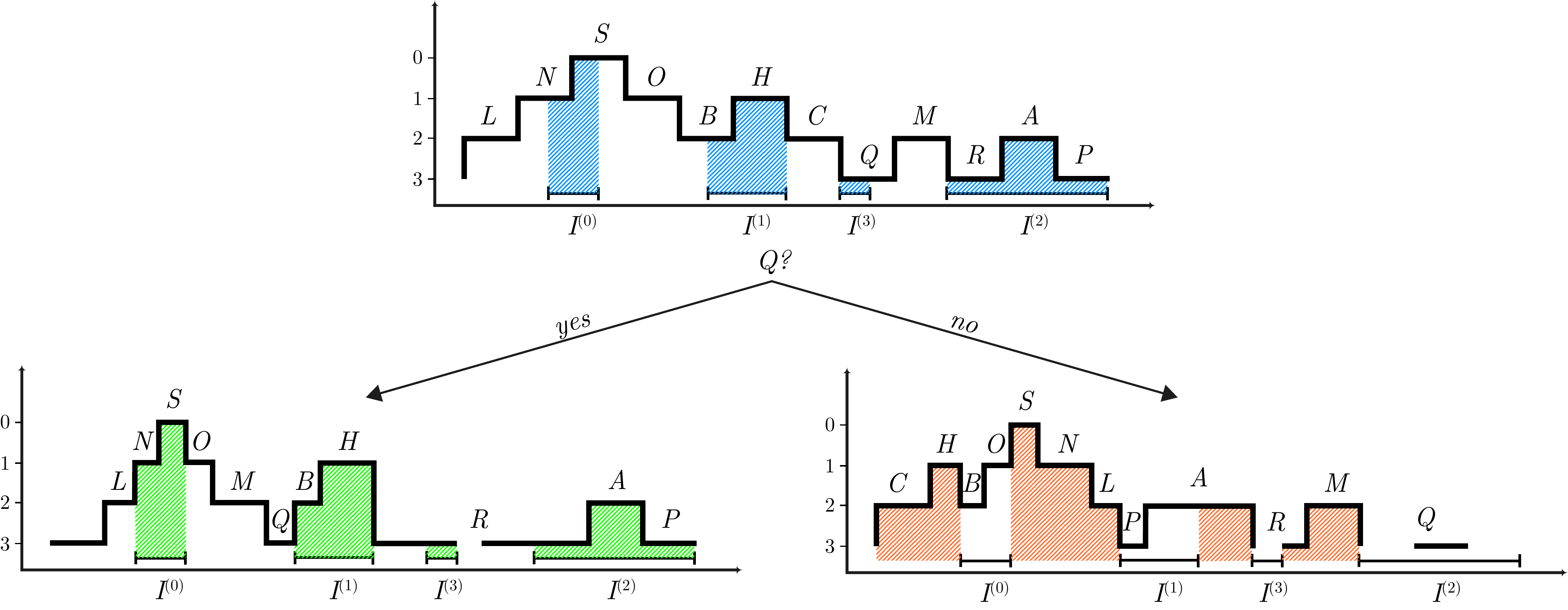}
	\caption{An Example of state dynamics. The question $Q$ is represented by the intervals $I^{(0)}, I^{(1)}, I^{(2)}$ and $I^{(3)}.$ The interval $I^{(0)}$ is {\em mode-splitting} at level $0$ and {\em downward step-splitting} at level $1$; $I^{(1)}$ is {\em mode-covering} at level $1$ and {\em saddle-splitting} at level $2$; $I^{(2)}$ is {\em mode-covering} at level $2$ and {\em saddle-covering} at level $3$; $I^{(3)}$ is {\em saddle-splitting} at level $3$. In the resulting states the filled volumes indicate the arcs of the state remained unchanged. The $\sigma^{yes}, \sigma^{no}$ states are represented on the support of the original state $\sigma$ to show how the elements belonging the lowest level disappear (the blank gaps on the shapes) from the support when they are in contradiction with more than 3 answers.}
		\label{fig.well-shaped-dynamic}
\end{figure}

We will define conditions on the intersection between the intervals defining a question and the arcs of a (well-shaped) 
state $\sigma$ such that the both $\sigma^{yes}$ and $\sigma^{no}$ are well shaped. It turns out that this condition can be defined 
locally, level by level and arc by arc.  

Let $\sigma$ be a well shaped state. 
Let ${\cal L}^{\sigma}(i)$ denote the list of arcs on level $i$ of $\sigma$. Let 
$|{\cal L}^{\sigma}(i)|$ denote the number of arcs on level $i$ of $\sigma$. Let $a$ be an arc on level $i$ of $\sigma$. Let 
$a_L$ and $a_R$ be the arcs preceding and following $a$ in ${\cal L}^{\sigma}.$
Let $I$ be an interval with non-empty intersection with the arc $a$.  
Assume a question $Q =I$ is asked and 
as a result of the answer we only update the level of the elements in $a$, 
while we leave the elements of $a_L$ and $a_R$ on the level they were before the answer.  
We denote by $\delta(a, I, yes, \ell)$ (respectively  $\delta(a, I, no, \ell)$) the difference between the number of arcs on level 
$\ell$ containing the elements of $a_L, a, a_R$ before and after the answer---when we only update the level of elements in $a$.

Note that an arc $a$ on level $i$ of $\sigma$ before the answer, can only be an arc on level $i$ or $i+1$ after the answer, 
since no arc can increase their level by more than one, and no arc can decrease their level. 
Thus, for $\ell \neq i,i-1$ we have that $\delta(a, I, yes, \ell), \delta(a, I, no, \ell) = 0.$ 
Moreover, we also have that for $\ell \in \{i, i+1\}$, it holds that  $\delta(a, I, yes, \ell), \delta(a, I, no, \ell) \in \{-1, 0, 1\}.$ 

The importance of the above definition is that it can be employed to characterize the local evolution 
of the arcs on the different levels of a well shaped state. 
For instance, for an interval $I$ that intersects both $a$ and $a_R,$ 
because the definition of $\delta(a,I, \cdot, \ell)$ only accounts for the change 
of level that happen in  the arc $a$ {\em assuming 
that both $a_R$ and $a_L$ are kept as they were before the answer}, it is not hard to check that   
$\delta(a, I, yes, \ell) + \delta(a_R, I, yes, \ell)$ correctly  counts the difference in the number of arcs spanning 
$a$ and $a_R$ on level $\ell$, from before to after the answer $yes$ to 
$I$---when we only consider the change of level of elements of $a$ and $a_R$.  
This modular behaviour of the $\delta()$ trivially extends to the case of intervals spanning more arcs. Then, we can 
compute the dynamics of the number of arcs on the level $i$ of the state before and after a 
question $Q$ by summing up over each interval $I$ of $Q$, the $\delta(a,I,yes,i)$ contribution on each arc intersecting $I$ and
 for each interval $I$ not in  $Q$ (i.e., in $\overline{Q}$), the $\delta(a,I, no, i)$ contribution on each arc covered by $I.$ 

Given  a well shaped  state $\sigma$ with $ {\cal L}^{\sigma}(i)$ as defined above, 
and a question $Q$ let $\sigma^{yes}$ and $\sigma^{no}$ be the resulting states for a yes and no answer, and  
let ${\cal L}^{\sigma}_{yes}(i), {\cal L}^{\sigma}_{no}(i)$ be 
defined in analogy to ${\cal L}^{\sigma}$ and ${\cal L}^{\sigma}(i)$.
Assuming that $\sigma^{-1}(-1) = \emptyset$, for each $i = 0, 1, 2, 3$ we have 
\begin{eqnarray} 
|{\cal L}^{\sigma}_{yes}(i)| = |{\cal L}^{\sigma}(i)| &+&  \sum_{I \in Q} \, \, \, 
\sum_{\substack{a \in {\cal L}^{\sigma}(i) \cup {\cal L}^{\sigma}(i-1)\\ a \cap I \neq \emptyset}} \delta(a, I, yes, i) \nonumber \\
&+& 
\sum_{I \in \overline{Q}}  \, \, \, 
\sum_{\substack{a \in \sigma^{-1}(i) \cup \sigma^{-1}(i-1) \\ a \subseteq I}} \delta(a, I, no, i)   \label{eq:dyn1}
\end{eqnarray}
\begin{eqnarray}
|{\cal L}^{\sigma}_{no}(i)| = |{\cal L}^{\sigma}(i)| &+& \sum_{I \in Q} \, \, \, 
\sum_{\substack{a \in \sigma^{-1}(i) \cup \sigma^{-1}(i-1) \\ a \cap I \neq \emptyset}} \delta(a, I, no, i)   \nonumber \\
&+& 
\sum_{I \in \overline{Q}}  \, \, \, 
\sum_{\substack{a \in \sigma^{-1}(i) \cup \sigma^{-1}(i-1) \\ a \subseteq I}} \delta(a, I, yes, i)  \label{eq:dyn2}
\end{eqnarray}
The following proposition,  which is easily verified by direct inspection, summarizes the values of $\delta()$ which are significative for our
analysis.

\begin{proposition}\label{prop:question_effect}
Let $\sigma$ be a state and $a$ be a non empty arc of $\sigma^{-1}(i)$. 
Let $Q$ be a question and $I$ an interval of $Q$ with non-empty intersection with $a$. Then we have
\begin{enumerate}
	\item[a)] if $Q$ is \emph{ saddle-splitting } at level $i$, then for $\ell = i, i+1\,$ we have $\delta(a, I, yes, \ell) = \delta(a, I, no, \ell) = 1$, 
	i.e., the number of arcs is increased by 1 on level $i$ and $i+1$ both in the case of a yes and a no answer, since part of the saddle 
	is transferred to the next level and therefore in the list of arcs there will be an additional (empty) arc between the part going to level $i+1$ and the flanking arc at level $i-1.$
	
	\item[b)] if $Q$ is  \emph{ mode-splitting } at level $i$, then for $\ell = i, i+1\,$ we have $\delta(a, I, yes, \ell) = \delta(a, I, no, \ell) = 0$, i.e., 
	the number of arcs remains unchanged on level $i$ and $i+1$ both in the case of a yes and a no answer, since the part of the 
	mode that is transferred to level $i+1$ get merged with the flanking arc on level $i+1$--- recall that the value of 
	$\delta$ is defined on the assumption that no other arc changes level.

	\item[c)] if $Q$ is  \emph{ upward step-splitting } at level $i$, then for $\ell = i, i+1\,$ we have $\delta(a, I, yes, \ell) = 0$ and $ \delta(a, I, no, \ell) = 1$, i.e., the number of arcs is increased by 1 on level $i$ and $i+1$ only in the case of a no answer, since part of the step is transferred to the next level and therefore in the list of arcs there will be an additional (empty) arc between the part going to level $i+1$ and the flanking arc at level $i-1.$
	
	\item[d)] if $Q$ is  \emph{ downward step-splitting } at level $i$, then for $\ell = i, i+1\,$ we have $\delta(a, I, yes, \ell) = 1$ and $ \delta(a, I, no, \ell) = 0$, i.e., the number of arcs is increased by 1 on level $i$ and $i+1$ only in the case of a yes answer, since part of the step is transferred to the next level and therefore in the list of arcs there will be an additional (empty) arc between the part going to level $i+1$ and the flanking arc at level $i-1.$

	\item[e)] if $Q$ is  \emph{ saddle-covering } at level $i$, then for $\ell = i, i+1\,$ we have $\delta(a, I, yes, \ell) = 0$ and $ \delta(a, I, no,\ell) = 1$, i.e., the number of arcs is increased by 1 on level $i$ and $i+1$ only in the case of a no answer, since the saddle is transferred to the next level and therefore in the list of arcs there will be an additional (empty) arc between the saddle going to level $i+1$ and one of the flanking arc at level $i-1.$ --- note that the other empty needed arc is the saddle arc becoming empty at level $i.$

	\item[f)] if $Q$ is  \emph{ mode-covering } at level $i$, then for $\ell = i, i+1\,$ we have $\delta(a, I, yes, \ell) = 0$ and $ \delta(a, I, no, \ell) = -1$, i.e., the number of arcs is decreased by 1 on level $i$ and $i+1$ only in the case of a no answer, since the mode is transferred to level $i+1$ get merged with both the flanking arcs at level $i+1$  into one single arc at level $i+1.$
	
	\item[g)] if $Q$ is  \emph{ step-covering } at level $i$, then for $\ell = i, i+1\,$ we have $\delta(a, I, yes, \ell) = \delta(a, I, no, \ell) = 0$, i.e., the number of arcs remains unchanged on level $i$ and $i+1$ both in the case of a yes and a no answer, since the arc (which was a step at level $i$) is transferred to level $i+1$ and gets merged with the flanking arc on level $i+1$ and at level $i$ we have a new empty arc (to satisfy the unit increase in the level of adjacent arcs).

\end{enumerate}
For the complementary question $\overline{Q}$,  the same rules apply with the role of $yes$ and $no$ swapped.
\end{proposition}

By Proposition \ref{prop:question_effect}, it follows that, given a well shaped state $\sigma$ and a question $Q$, states $\sigma^{yes}$ and $\sigma^{no}$ 
are both well shaped if and only if for each $i=0,1,2,3$ we have $|{\cal L}^\sigma (i)| = |{\cal L}^\sigma_{yes} (i)| = |{\cal L}^\sigma_{no} (i)|,$ 
which, by (\ref{eq:dyn1})-(\ref{eq:dyn2}),  is equivalent to 
have 

\begin{eqnarray}
\sum_{I \in Q}  \, \, \,
\sum_{\substack{a \in {\cal L}^{\sigma}(i) \cup {\cal L}^{\sigma}(i-1)\\ a \cap I \neq \emptyset}} \!\!\!\!\!\! \delta(a, I, yes, i)+ 
\sum_{I \in \overline{Q}}  \, \, \, 
\sum_{\substack{a \in {\cal L}^{\sigma}(i) \cup {\cal L}^{\sigma}(i-1) \\ a \subseteq I}} \!\!\!\!\!\!\delta(a, I, no, i)   = 0 \label{eq:well-shapeness-condition-yes}\\
\sum_{I \in Q}  \, \, \,
\sum_{\substack{a \in {\cal L}^{\sigma}(i) \cup {\cal L}^{\sigma}(i-1) \\ a \cap I \neq \emptyset}} \!\!\!\!\!\! \delta(a, I, no, i) + 
\sum_{I \in \overline{Q}}  \, \, \, 
\sum_{\substack{a \in {\cal L}^{\sigma}(i) \cup {\cal L}^{\sigma}(i-1) \\ a \subseteq I}} \!\!\!\!\!\!\delta(a, I, yes, i)  = 0. \label{eq:well-shapeness-condition-no}
\end{eqnarray}

Using Proposition \ref{prop:question_effect} and the condition in (\ref{eq:well-shapeness-condition-yes})-(\ref{eq:well-shapeness-condition-no}),
 we have the following sufficient conditions for build a question that preserves well-shapeness.

\begin{lemma} \label{lemma:well-shapeness}
Given a  well-shaped state $\sigma$ and a question $Q$, if the following set of conditions is satisfied then both the resulting states 
$\sigma_{yes}$ and $\sigma_{no}$ are 
well-shaped. 
For each $i = 0, 1, 2$ 
\begin{itemize}
\item[(a)] At most one arc is splitted on level $i$
\item[(b)] Exactly one of the following holds 
	\begin{enumerate}
	\item[(i)]   at level $i$ the question $Q$ is mode-splitting;
	\item[(ii)]  at level $i$ the question $Q$ is upward step-splitting and mode-covering.
	\item[(iii)] at level $i$ the question $Q$ is downward step-splitting and a mode is completely uncovered---equivalently the complementary question $\overline{Q}$ is mode-covering at level $i$. 
	\item[(iv)]  at level $i$ the question $Q$ is saddle-splitting and mode-covering and there is also a mode completely uncovered---equivalently the complementary question $\overline{Q}$ is mode-covering at level $i$.
	\end{enumerate}
\item[(c)] If besides the condition in (b) the question $Q$ is also saddle-covering then $Q$ also covers a mode different from the one possibly used to satisfy (b).
\end{itemize}

\end{lemma}

\begin{proof}
Let $a \in \sigma^{-1}(i)$ the splitted arc at level $i$. Let $I$ the interval of $Q$ with non-empty intersection with $a.$ Then 
\begin{itemize}
	\item[i)] if $Q$ is \emph{ mode-splitting } at level $i$, then for $\ell = i, i+1\,$ we have that $\delta(a,I,yes,\ell) = \delta(a,I,no,\ell) = 0$.
	\item[ii)] if $Q$ is \emph{ upward step-splitting } at level $i$, then for $\ell = i, i+1\,$ we have $\delta(a,I,yes,\ell) = 0$ and $\delta(a,I,no,\ell) = 1$. Moreover, let 
	$ b \in \sigma^{-1}(i)$ the \emph{entairely covered mode} at level $i$ and $I_b$ is the interval of $Q$ witch cover $b$, then for $\ell = i, i+1\,$
	we have that $\delta(b,I_b,no,\ell) = -1$.
	
	 Summing up level by level all the $\delta()$ elements, for $\ell = i, i+1\,$ we have that $\delta(a,I,no,\ell) + \delta(b,I_b,no,\ell) = 0$.
	 
	\item[iii)] if $Q$ is \emph{ downward step-splitting } at level $i$, then for $\ell = i, i+1\,$ we have that $\delta(a,I,yes,\ell) = 1$ and $\delta(a,I,no,\ell) = 0$. 
	Moreover, let $ b \in \sigma^{-1}(i)$ the \emph{completely uncovered mode} at level $i$ and $I_b$ is the interval of $\overline{Q}$ 
	witch cover $b$, then for $\ell = i, i+1\,$ we have that $\delta(b,I_b,no,\ell) = -1$.
	
	 Summing up level by level all the $\delta()$ elements, for $\ell = i, i+1\,$ we have that $\delta(a,I,yes,\ell) + \delta(b,I_b,no,\ell) = 0$.
	 
	\item[iv)] if $Q$ is \emph{ saddle step-splitting } at level $i$, then for $\ell = i, i+1\,$ we have that $\delta(a,I,yes,\ell) = \delta(a,I,no,\ell) =  1$. 
	Moreover, let $ b \in \sigma^{-1}(i)$ the \emph{entairely covered mode} at level $i$ and $I_b$ is the interval of $Q$ witch cover $b$, 
	then for $\ell = i, i+1\,$ we have that $\delta(b,I_b,no,\ell) = -1$. 
	Also let $ c \in \sigma^{-1}(i)$ the \emph{completely uncovered mode} at level $i$ and $I_c$ the interval of $\overline{Q}$
	 witch cover $c$, then for $\ell = i, i+1\,$ we have that $\delta(c,I_c,no,\ell) = -1$.
	
	Summing up level by level all the $\delta()$ elements, for $\ell = i, i+1\,$ the result is that 
	$\delta(a,I,yes,\ell) + \delta(c,I_c,no,\ell) = 0$ and $\delta(a,I,no,\ell) + \delta(b,I_b,no,\ell) = 0$.

\end{itemize}

Moreover, let $ b \in \sigma^{-1}(i)$  be a \emph{saddle} at level $i$ and $I_b$ is the interval of $\overline{Q}$ which covers $b$, 
then for $\ell = i, i+1\,$ we have that $\delta(a,I,yes,\ell) = 0 $ and $ \delta(a,I,no,\ell) = 1$. 
Let $ c \in \sigma^{-1}(i)$ a \emph{mode} at level $i$ ---recall that $c$ is different from the one possibly used to satisfy the previous cases --- 
and $I_c$ the interval of $Q$ witch cover $c$, then for $\ell = i, i+1\,$ we have that $\delta(c,I_c,no,\ell) = -1$.
Summing up all the  $\delta()$ elements we have that $\delta(a,I,no,\ell) + \delta(c,I_c,no,\ell) = 0$ for $\ell = i, i+1\,$.

Following the previous conditions both the equations (\ref{eq:well-shapeness-condition-yes}) and 
(\ref{eq:well-shapeness-condition-no}) are satisfied. 
Then the resulting states $\sigma_{yes}$ and $\sigma_{no}$ are well shaped. \qed

\end{proof}

Now we are ready to prove our main technical result.

{\bf Theorem \ref{th:generalized-questions}.}
{\em Let $\sigma$ be a well-shaped state of type $\tau (\sigma) = (a_0, b_0,c_0,d_0) $.
For all integers 
$\displaystyle{ 0 \leq a \leq a_0 \quad  0 \leq  b \leq \lceil \frac{1}{2} b_0 \rceil 
\quad 0 \leq  c \leq \lceil \frac{1}{2} c_0 \rceil  
\quad 0 \leq  d \leq \lceil \frac{2}{3} d_0 \rceil
}$
there exists a $4$-interval question $Q$ of type $|Q| = [a,b,c,d]$  that we can ask in state $\sigma$ and such 
that both the resulting "yes" and "no" states are \emph{well-shaped}.}

\begin{proof}
We first show how to select the intervals of the question $Q$ in order to satisfy the desired type.
We proceed level by level. For each $i=0,1,2,3,$ we show how to select up to 4 intervals that cover the 
required number of elements in the first $i$ levels.  
For each level $i = 0,1,2,3$ we record in a set $\mathcal{E}(i)$ the extremes of the 
intervals selected so far that have a neighbour on  the next level. We refer to the elements in $\mathcal{E}(i)$ as the boundaries at level $i$. When processing the next level, we try to select arc neighbouring the elements in $\mathcal{E}(i)$ since this means we can cover elements at the new level without using additional intervals. Arguing with respect to such boundaries, 
we show that the (sub)intervals selected at all level can be merged into at most $4$ intervals. Hence the resulting question $Q$ is a $4$-interval question.
Finally, we will show that asking $Q$ in  $\sigma$ both the resulting states are well-shaped states. 

Recall the arc notation used in (\ref{eq:well shaped})-(\ref{eq:well shaped2}). 
In our construction, a special role will be played by the arcs $S, H, A$, which are the greatest mode respectively of level $0,1$ and $2$,  and the larger between their two neighbouring arcs at the level immediately below. Therefore, let us  denote by $A^{(1)}$ the larger arc between $N$ and $O$; we denote by $A^{(2)}$ be the larger arc between $B$ and $C$; and 
finally, we denote by $A^{(3)}$ the larger arc between $R$ and $P$.

Moreover, we denote by $s^+$ the boundary between $S$ and $A^{(1)}$ and with $s^-$ the other boundary of $S$.

Analogously, we denote by $h^+$ the boundary between $H$ and $A^{(2)}$ and with $h^-$ the other boundary of $H$.

We denote by $a^+$ the boundary between $A$ and $A^{(3)}$ and with $a^-$ the other boundary of $A$. 

\begin{enumerate}
\item[Level $0$] For any  $0 < a \leq a_0$ there exists $s^{\ast} \in S $  such that denoting by $S^\ast$ the sub-arc of $S$ between $s^\ast$ and $s^+$
we have that $|S^\ast | = a $ 
and the boundary of $S^\ast$ includes $s^+.$ Then we have ${\cal E}(0) \supseteq \{ s^+\}.$

Therefore, with one interval $I^{(0)}= S^\ast$ we can accommodate the $a$ elements on level $0$ and guarantee that this interval has an extreme at $s^+.$

Moreover, the interval $I^{(0)}$ on arc $S$ is a mode-splitting interval.

\item[Level $1$] By the assumption $b \leq \lceil \frac{1}{2} b_0 \rceil $, and the definition of $A^{(1)}$ it follows that 
$b \leq |A^{(1)}| + |H|$. We now argue by cases
\begin{enumerate}
\item $b \leq |H|$. Then there exists $h^\ast$ in $H$ such that the sub-arc $H^\ast \subseteq H$ between $h^\ast$ and $h^+$ satisfies 
$|H^*| = b$ and we can cover it with one interval $I^{(1)}$ with a boundary at $h^{+}$.
\item $|H| < b \leq |H|+|A^{(1)}|$. Then, there exists $x_1^* \in A^{(1)}$ such that  letting $X^{(1)}$ be the sub-arc of $A^{(1)}$  
between $x_1^*$ and $s^{+}$, we have $|X^{(1)}| + |H| = b$ and we can cover these $b$ elements extending the previously defined 
$I^{(0)}$ so that it covers $X^{(1)}$ too and having $I^{(1)} = H.$ In this case we have that the boundaries of $I^{(1)}$ are both $h^+$ and 
$h^-.$
\end{enumerate}
		
\noindent
Summarising, we can cover the $a$ elements on level $0$ and the $b$ elements on level $1$ with at most two intervals and
guarantee that the boundaries of these intervals include $h^+.$ 

Moreover either the interval $I^{(1)}$ on arc $H$ is a mode-splitting interval or the interval $I^{(1)}$ covers entirely the mode $H$ and the interval $I^{(0)}$ on arc $A^{(1)}$ is a step-splitting interval.

Then, the choice of the intervals so far satisfies conditions in  Lemma \ref{lemma:well-shapeness}. 

\item[Level $2$] Again we argue by cases 
\begin{enumerate}
\item $ c \leq |A|.$ Then, there exists $a^\ast$ in $A$ such that letting $A^\ast$ be the sub-arc of $A$ between $a^\ast$ and $a^{+}$ we have 
$|A^\ast| = c$ and we can cover it with one interval $I^{(2)} = A^\ast$ with one boundary in $a^{+}$. 
\item  $|A| < c \leq |A| + |A^{(2)}|$ Then, there exists $x^* \in A^{(2)}$ such that  letting $X^{(2)}$ be the sub-arc of $A^{(2)}$  
between $x^*$ and $h^{+}$, we have $|X^{(2)}| + |A| = c$ and we can cover these $c$ elements extending the previously defined 
$I^{(1)}$ so that it covers $X^{(2)}$ too and having $I^{(2)} = A.$ In this case we have that the boundaries of $I^{(2)}$ are both $a^+$ and 
$a^-.$
\item $c > |A| + |A^{(2)}|.$ Let $E$ denote the largest arc on level 2 not in $\{A^{(2)}, A\}.$ Then, by the definition of $A^{(2)}$ we have that 
$|A| + |A^{(2)}| + |E| \geq |Y|+|Z|$ where $Y$ and $Z$ denote the arcs on level $2$ not in $\{A, A^{(2)},E\}.$  This is true because, at 
least one of the arcs $Y, Z$ is not larger than $A^{(2)}$ and the other one is not larger than $E$.

Then, by the assumption $c \leq \lceil \frac{c_0}{2} \rceil$ it follows that  $ |A| + |A^{(2)}|+ |E| \geq c.$ 
Let $z$ be the boundary of $A^{(2)}$ on the opposite side with respect to $H$.
Let $e^+$ be the boundary between $E$ and a neighbouring arc at level $3$ or at level $1$ according to whether 
$z$ is flanking an arc at level $1$ or an arc at level $3$---with reference to Figures \ref{fig.well-shaped}, \ref{fig.well-shaped2}, is
an easy direct inspection shows that such a choice is always possible. 

Therefore, there exists $e^\ast \in E$ such that letting   
$E^\ast$ be the sub-arc of $E$ between $e^\ast$ and $e^+$ we have  $|A| + |A^{(2)}| + |E^\ast |= c$ and 
we can cover the corresponding set of elements by: (i)  extending $I^{(1)}$ from $h^+$ and have it include the whole $A^{(2)}$; (ii) defining 
$I^{(2)} = A$; defining a fourth interval $I^{(3)} = E^*.$ Therefore, we have that the boundaries of $I^{(2)}$ are both $a^+$ and $a^-$ and, in the case 
of a $\sigma$ of type in Fig. \ref{fig.well-shaped2},  the 
boundary of $I^{(3)}$ includes $e^+,$ and the boundary of $I^{(1)}$ is the boundary $z$ 
of the arc $A^{(2)}$ where it joins its adjacent arc at level 3.  

\end{enumerate}

Summarising, we can cover the $a$ elements on level $0$ and the $b$ elements on level $1$ and the $c$ elements of level $2$ 
with at most four intervals. More precisely, if, proceeding as above we only use  three intervals, $I^{(0)}, I^{(1)}, I^{(2)},$ 
(and set $I^{(3)} = \emptyset$),  we also guarantee that the boundaries of these intervals
include $a^+$. 
On the other hand, if we use four intervals, (in particular,  $I^{(3)} \neq \emptyset$) 
we have that the boundaries of these intervals include $a^+, a^-$ and exactly one between $e^+, z$.
Therefore $\{a^+, a^-\} \subset {\cal E}(2) \subset \{a^+, a^-, z, e^+\}$.
Notice that, since there are only three arcs at level 3; in the case where $I^{(3)} \neq \emptyset$ 
either there is an arc on level $3$ with both ends neighbouring the boundaries in ${\cal E}(2),$ or each arc on level 3
has one end neighbouring a boundary in ${\cal E}(2).$
 
Moreover exactly one of the following cases holds (i) the interval $I^{(2)}$ on arc $A$ is a mode-splitting interval; 
(ii) the interval $I^{(2)}$ covers entirely the mode $H$ and the interval $I^{(1)}$ on arc $A^{(2)}$ is a step-splitting interval; 
(iii) the interval $I^{(2)}$ covers the mode $H$, no interval intersects mode $M$ and the interval $I^{(1)}$ on arc $A^{(2)}$ is
saddle-splitting; 
(iv) the interval $I^{(2)}$ covers  the mode $H$ and the interval $I^{(1)}$ covers the arc $A^{(2)}$ 
and the interval $I^{(3)}$ on arc $E$ is downward step-splitting; 
(v) the interval $I^{(2)}$ covers the mode $H$, the interval $I^{(1)}$ covers the arc $A^{(2)}$, hence it is saddle-covering, 
and the interval $I^{(3)}$ on arc $E$ is upward step-splitting.	.

In all the above five cases, the (partial) question built so far, with the intervals defined for levels 0, 1, 2, satisfy the conditions
of Lemma \ref{lemma:well-shapeness}.

\item[Level 3] Let us denote by $W, U$ the two arcs at level 3 which are different from $A^{(3)},$ with $|W| \geq |U|.$ 
Then, by definition we also have $|A^{(3)}| \geq |U|$, hence
$ |A^{(3)}| + |W| \geq \frac{2}{3} d_0 \geq d.$ We now argue by cases:
\begin{enumerate}
\item $d < |A^{(3)}|$. Then,   there exists $x^\ast_3$ in $A^{(3)}$ such that the sub-arc $X^{(3)}$ between $a^+$ and $x^\ast_3$ has cardinality 
$|X^{(3)}| = d$ and can be covered by extending $I^{(2)}$ (which have a boundary at $a^+$).
\item $|A^{(3)}| < d \leq  |A^{(3)}| + |W|$.  

We have two sub-cases:
\begin{itemize}
\item $I^{(3)} = \emptyset.$ I.e., for accommodating the question's type on Levels 0,1,2, we have only used three intervals. By assumption, there exists a sub-arc $W^\ast$ of $W$ such that $|W^\ast| + |A^{(3)}| = d.$ Then, defining $I^{(3)} = W^\ast,$ and extending $I^{(2)}$ (as in the previous case) so that 
it includes the whole of $A^{(3)}$ guarantees that the four intervals $I^{(0)}, \dots I^{(3)}$ define a question of the desired type. 
\item $I^{(3)} \neq \emptyset.$ Then, by the observations above, as a result of the construction on level 2, either there is an 
arc on level $3$ with both ends at a boundary in ${\cal E}(2)$ or each arc on level $3$ has a boundary in ${\cal E}(2).$ In the latter case, we can 
clearly extend the intervals $I^{(2)}$ and $I^{(3)}$ in order to cover $d$ elements on level 3.
In the former case, let $Z$ denote the arc with both ends at boundaries in ${\cal E}(2)$. %
If $|Z| \geq d$, we can simply extend $I^{(2)}$ and $I^{(3)}$ towards the internal part of $Z$ until they include $d$ elements of $Z$.
If $|Z| < d$ then we can extend $I^{(2)}$ so that it includes $Z$ and $I^{(3)}.$ 
\end{itemize}

\end{enumerate}

\end{enumerate}

Since the way intervals are extended on level 3 do not affect the arc covering and splitting on the previous level, we have that in all cases the resulting 4-interval question satisfies the conditions of Lemma \ref{lemma:well-shapeness}, which 
guarantees that both resulting states are well-shaped. The proof is complete. 
\qed
\end{proof}

Refer to Figures \ref{fig.well-shaped-questions} and \ref{fig.well-shaped-questions2} for a  
pictorial representation of the  4-intervals question construction in the proof of Theorem \ref{th:generalized-questions}.

\begin{figure*}[h]
	\centering \includegraphics[width=\textwidth]{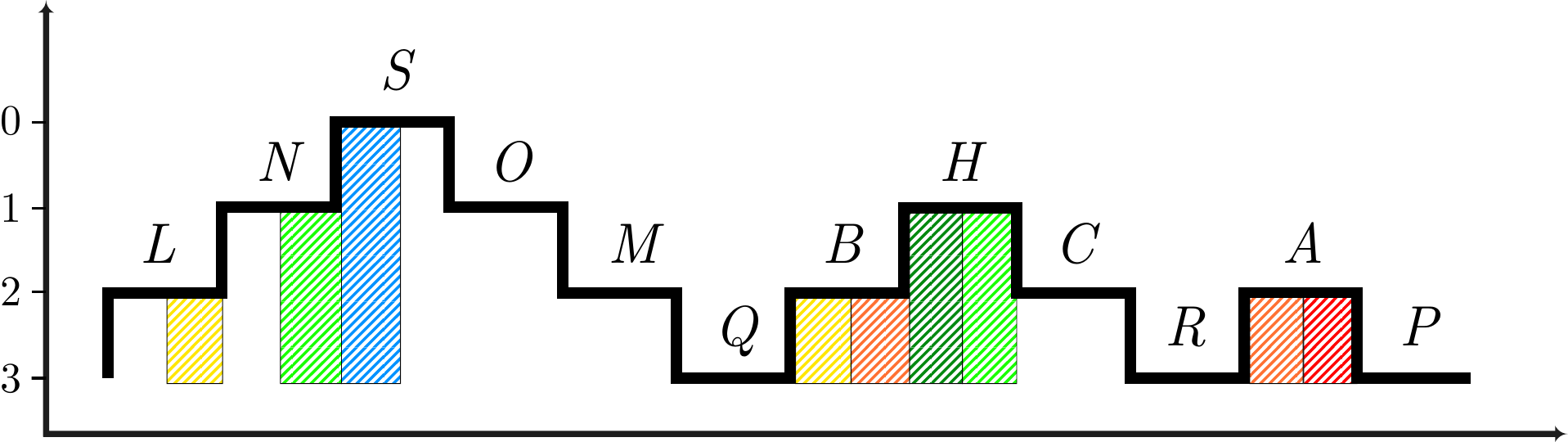}
	\caption{A well-shaped state like in (\ref{eq:well shaped}) and the cuts of a 4 interval question.}
	\label{fig.well-shaped-questions}
\end{figure*}

\begin{figure*}[h]
	\centering \includegraphics[width=\textwidth]{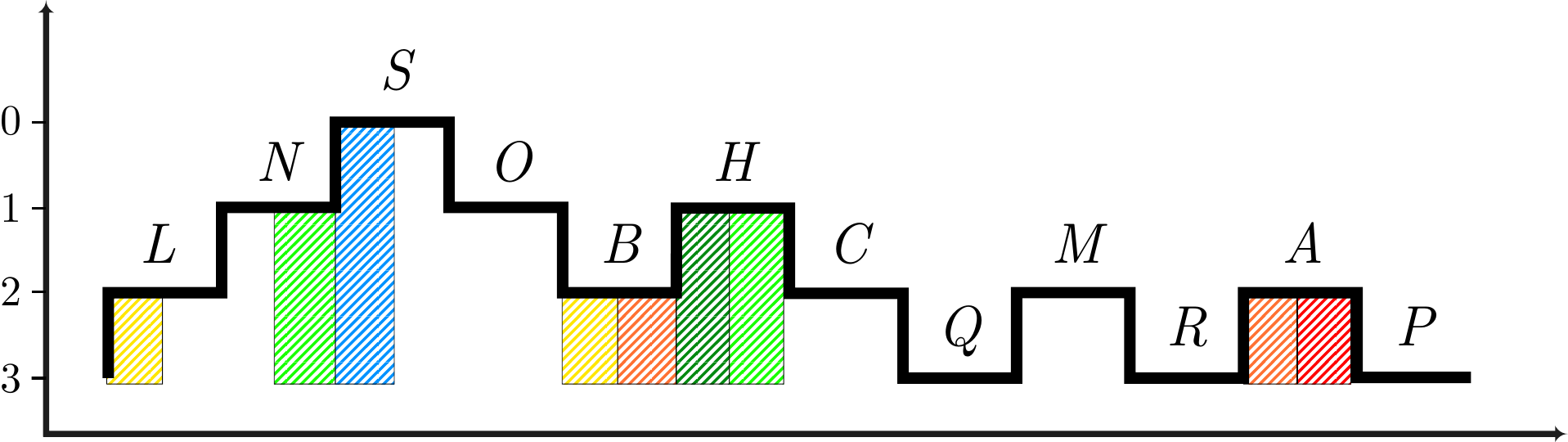}
	\caption{A well-shaped state like in (\ref{eq:well shaped2}) and the cuts of a 4 interval question.}
	\label{fig.well-shaped-questions2}
\end{figure*}

We assume the following relative order on the arcs' sizes. 
On level $1$ we assume $N \geq O$ and on level $2$ we assume $B \geq C$, $L \geq M$ and $A \geq M$.
The questions are depicted for both the feasible well-shaped states for a $3$ lies game.
Then on level $0$ the arc $S$ is split (the light blue question), on level $1$ either $H$ is split (the dark green question) or $N$ is split (the dark green and the light green question) and $H$ is covered entirely. On level $2$ one of the following holds, either $A$ 
is split (the red question) or $B$ is split (the orange and red question) and the mode $A$ is covered entirely, 
or $L$ is split (the yellow, orange and red question) and the mode $A$ is covered entirely.
In order to guarantee the well-shapeness preservation, note that in the questions depicted in figure \ref{fig.well-shaped-questions2} 
on level $2$ the mode $M$ is entirely uncovered, moreover the interval splitting arc $L$ is upward step-splitting in  
Figure \ref{fig.well-shaped-questions} and downward step-splitting  in Figure \ref{fig.well-shaped-questions2}.

\section{The non asymptotic strategy}
In this section we show how to employ the machinery of Lemma \ref{lemma:well-shapeness} to 
obtain an exact (non-asymptotic) characterization of the instances of the Ulam-R\'enyi game with $3$ lies on 
a serach space of cardinality $2^m$ that admit a perfect strategy only using $4$-interval questions. 
For this we will exploit the perfect strategies of \cite{negro1992ulam} and show that they can be implemented  using only $4$-interval questions. 
We are going to prove the following result.
\begin{theorem}\label{th:non-asymptotic}
For all $m \in  \mathbb{N} \setminus \{2,3,5\}$ in the game played over the search space $ \{0, \cdots, 2^m-1\}$ with $3$ lies, 
there exists a perfect 4-intervals strategy.
\end{theorem}

\noindent
{\bf Some technical lemmas.}
We start by proving some variants of Theorem \ref{th:generalized-questions} that can be employed in particular cases where 
Theorem \ref{th:generalized-questions} does not directly apply.

\begin{lemma}\label{th:specialized-guzicki}
Let $\sigma$ be a well-shaped state of type $\tau (\sigma) = (0,b_0,c_0,d_0) $. 
For all integers 
$\displaystyle{0 \leq  b \leq \lceil \frac{1}{2} b_0 \rceil
\quad 0 \leq  c \leq \lceil \frac{1}{2} c_0 \rceil  
\quad 0 \leq  d \leq d_0 
}$
there exists a $4$-interval question $Q$ of type $|Q| = [0,b,c,d]$ that can be asked in state $\sigma$ and such 
that both the resulting "yes" and "no" states are \emph{well-shaped}.
\end{lemma}

\begin{proof}

The proof structure is analogous to the proof of Theorem \ref{th:generalized-questions}
\begin{enumerate}
\item[Level $1$] By the assumption $b \leq \lceil \frac{1}{2} b_0 \rceil $ and definition of $H$ as the greatest mode of level $1$ it follows that $b \leq |H|$.
Then there exists $h^\ast$ in $H$ such that the sub-arc $H^\ast \subseteq H$ between $h^\ast$ and $h^+$ satisfies 
$|H^*| = b$ and we can cover it with one interval $I^{(0)}$ with a boundary at $h^{+}$. Then we have ${\cal E}(1) \supseteq \{ h^+\}.$

Summarising, we can cover the $b$ elements on level $1$  with at most one interval and
guarantees  that the boundaries of these intervals include $h^+.$ 

Moreover the interval $I^{(0)}$ on arc $H$ is a mode-splitting interval.

\item[Level $2$] 
As showed in Theorem \ref{th:generalized-questions}, we can cover the $b$ elements on level $1$ and the $c$ elements on level $2$, with at most three intervals.
Moreover, if, proceeding as in Theorem \ref{th:generalized-questions} we only use two intervals, $I^{(0)}, I^{(1)},$ (and set $I^{(2)} = \emptyset$) we also guarantee that the boundaries of these intervals include $a^+$. 
Otherwise, if we use three intervals, we have that the boundaries of those intervals include $a^+, a^-$ and exactly one between $e^+, z$.

\item[Level 3]
Let us denote by $W, U$ the two arcs at level 3 which are different from $A^{(3)}.$
We now argue by cases:
\begin{enumerate}
\item $I^{(2)} = \emptyset.$
We show how  to cover up to $d_0$ elements on level 3. Extending $I^{(1)}$ we can cover as much as we need from the arc $A^{(3)}$, 
after that we can use the two remaining intervals $I^{(2)}$ and $I^{(3)}$ to cover, respectively $W$ and $U$.
This guarantees that with at most four intervals $I^{(0)}, \dots I^{(3)}$ we have a question of the desired type. 
\item $I^{(2)} \neq \emptyset.$
Then, by the observations above, both the boundaries include $a^+$ and $a^-$. As before, Extending $I^{(1)}$ we can cover as much as we need from the arc $A^{(3)}$ and the other neighbouring arc of $A$, say $W$. Finally, we can use the remaining interval $I^{(3)}$ to cover the last uncovered arc $U$. 
\end{enumerate}
\end{enumerate}

In all cases
the resulting 4-interval question satisfies the Lemma \ref{lemma:well-shapeness} 
conditions, which guarantees that both resulting states are well-shaped.
\qed
\end{proof}%

\begin{lemma}
\label{th:specialized-YES1}
Let $\sigma$ be a well-shaped state of type $\tau (\sigma) = (1,b_1,c_1,d_1) $. 
For all integers  
$\displaystyle{\lfloor \frac{1}{4} c_1 \rfloor \leq  c \leq \lceil \frac{1}{2} c_1 \rceil 
\quad 0 \leq  d \leq d_1} $
there exists a $4$-interval question $Q$ of type $|Q| = [1,1,c,d]$  that we can ask in state $\sigma$ and such 
that both the resulting "yes" and "no" states are \emph{well-shaped}.
\end{lemma}
\begin{proof}

Let $U$ be the smallest arc on level 2, 
different from $A$ and $B$. Then, assuming that $|A| \geq \lfloor \frac{1}{4} c_1 \rfloor$, 
we have that $ |U| \leq \lfloor \frac{1}{4} c_1 \rfloor $. 
This is true because, by the minimality of $U$ we have  
$\displaystyle{|U| \leq  \frac{\lfloor(1-\frac{1}{4})\rfloor }{3}c_1 \leq \lfloor \frac{1}{4} c_1 \rfloor.}$

Let $V$ be the neighbouring arc of $U$ at level 1. 
We denote with $v^-$ the boundary between $U$ and $V$ and with 
$u^+$ the other boundary of $U$.
The existence and uniqueness of $u^+$ is guaranteed by the definition of $U$ 
that excludes that  it is a saddle.

This proof has a different scheme with respect to the proof structure of Theorem \ref{th:generalized-questions}. 
In fact, the interval used to accommodate the element on level 1, is defined after the evaluation of the intervals at level 2.

\begin{enumerate}
	\item[Level $0$] Is treated as in Theorem \ref{th:generalized-questions}, then we have one element of level 0 that is accommodated by interval $I^{(0)}$. The interval on arc $S$ is a mode-covering interval.

	\item[Level $2$] We argue by cases 
	\begin{enumerate}
		\item $ \lfloor \frac{1}{4} c_1 \rfloor \leq c \leq |A|.$ 
		Then, there exists $a^\ast$ in $A$ such that letting $A^\ast$ be the sub-arc of $A$ between $a^\ast$ and $a^{+}$ we have $|A^\ast| + |U| = c$.
		
		We can cover it using the interval $I^{(1)} = U $ with one boundary in $u^{+}$, and with one other interval $I^{(2)} = A^\ast$ with one boundary in $a^{+}$. 
		
		\item $ |A| < c  \leq \lceil \frac{1}{2} c_1 \rceil. $ 
		Again we argue by cases on the size of the mode $H$ at level 1. In particular, we have that either the mode $H$ contains at least one element, or the mode $H$ is empty.
		In the former case, we can proceed as in the proof of Theorem \ref{th:generalized-questions} guaranteeing that the boundary $h^{+} $ is included in the interval $I^{(1)}$. Moreover, if we use tree intervals, we guarantee that the interval $I^{(2)}$ has the boundaries $a^+$ and $a^-$, and if we use four intervals, we guarantee that the intervals $I^{(1)}$ and $I^{(3)}$ have exactly one boundary between $e^+$ and $z$.
		In the latter case, we can argue by cases
		\begin{itemize}
			\item $ \lfloor \frac{1}{4} c_1 \rfloor \leq c \leq |A| + |U|.$ Then, there exists $x^* \in U$ such that  letting $X^{(2)}$ be the sub-arc of $U$  
			between $x^*$ and $u^{-}$, we have $|X^{(2)}| + |A| = c$ and we can cover these $c$ elements using the interval 
			$I^{(1)}$ to covers $X^{(2)}$ and having $I^{(2)} = A.$ In this case we have that the boundaries of $I^{(2)}$ are both $a^+$ and 
			$a^-.$
			\item $c > |A| + |A^{(2)}|.$ Let $E$ denote the largest arc on level 2 not in $\{U, A\}.$ Note that, since the mode at level 1 is empty we can glue together the arcs $B$ and $C$ flanking the mode. Then, we have that 
			$|A| + |U| + |E| \geq |Y|$ where $Y$ is the arc on level $2$ not in $\{A,U,E\}.$  This is true because, $E \geq Y$.
			Therefore, there exists $e^\ast \in E$ such that letting   
			$E^\ast$ be the sub-arc of $E$ between $e^\ast$ and $e^+$ we have  $|A| + |U| + |E^\ast |= c$ and 
			we can cover the corresponding set of elements by: (i)  using $I^{(1)}$ to cover the whole $U$; (ii) defining 
			$I^{(2)} = A$; defining a fourth interval $I^{(3)} = E^*.$ Therefore, we have that the boundaries of $I^{(2)}$ are both $a^+$ and $a^-$ and, the 
			boundary of $I^{(3)}$ includes $e^+,$ and the boundary of $I^{(1)}$ is the boundary $u^+$ 
			of the arc $U$ where this joins its adjacent arc at level 3.  
			
		\end{itemize}

	\end{enumerate}
	
	Summarizing, we can cover one element on level $0$ and the $c$ elements of level $2$ with at most four intervals. 
	More precisely, the main difference with the analysis on Theorem \ref{th:generalized-questions}, is that if we only use three intervals, $I^{(0)}, I^{(1)}, I^{(2)},$ (and set $I^{(3)} = \emptyset$),  we also guarantee that the boundaries of these intervals include $a^+$ and one between $a^-$ and $u^+$. Conversely, if we use four intervals, we have that the boundaries of these intervals include $a^+, a^-$ and exactly one between $e^+, z$.
	Thus, either $\{a^+, u^+\} \subset {\cal E}(2)$ or $\{a^+, a^-\} \subset {\cal E}(2) \subset \{a^+, a^-, z, e^+\}$.
	Indeed, if the mode $H$ is empty, and we are using four intervals, we have that the boundaries of these intervals include $a^+, a^-, e^+ $ and $u^+$. Thus we have that $\{a^+, a^-, e^+, u^+\} \subset {\cal E}(2)$. 

	Moreover, the analysis of the splitting intervals is the same as the cases on Theorem \ref{th:generalized-questions}, in addition we have that in the first case, the interval $I^{(1)}$ covers the arc $U$ and the interval $I^{(2)}$ on arc $A^\ast$ is a mode-splitting interval.
	
	\item[Level $1$]
	
	From the previous arguments, we have that one boundary between $u^-$ and $h^+$ is included in the interval $I^{(1)}$. Then, if $I^{(1)}$ include the boundary $u^-$ then we can extend it from $u^-$ up to $v^\ast$ in $V$ such that letting $X^\ast$ be a sub arc of $V$ between $v^\ast$ and $u^-$ with $|X^\ast| = 1.$ Otherwise, if the boundary included in $I^{(1)}$ is $h^+$ then we can extend $I^{(1)}$ from $u^-$ up to $h^\ast$ in $H$ such that letting $X^\ast$ be a sub arc of $H$ between $h^\ast$ and $u^-$ with $|X^\ast| = 1.$
	Then, we can cover one element on level $0$, one element on level $1$ and the $c$ 
	elements of level $2$ with at most four intervals. And the boundaries at level $3$ remain unchanged.
	Moreover, the interval  $I^{(1)}$ covers the arc $X^\ast$ and is mode-splitting if the arc $X$ is the mode $H$, otherwise is downward-step splitting.
	
	\item[Level 3] Let us denote by $W, Z$ the two arcs at level 3 which are different from $A^{(3)},$ with $|W| \geq |Z|.$ 
	\begin{enumerate}
		\item $ d \leq  |A^{(3)}| + |W|$.  
		This case is handled like in Theorem \ref{th:generalized-questions}.
		
		\item $|A^{(3)}| + |W| < d $.
		There exists a sub-arc $Z^\ast$ of $Z$ such that $|Z^\ast| + |W| + |A^{(3)}| = d.$ 
		As a result of the construction used on the previous level, we have that ${\cal E}(2)$ has at least two boundaries. 
		One boundary is one end of the arc $A^{(3)}$.
		The other boundary in ${\cal E}(2)$ is one end of the arc $X \in \{A^{(3)},W,Z\}$. 
		We have two sub-cases:
		\begin{itemize}
			\item $I^{(3)} = \emptyset.$ Then we have that either $\{a^+,a^-\} \subseteq {\cal E} (2)$ or $\{a^+,u^-\} \subseteq {\cal E} (2)$. 
			Moreover, either there is an arc on level $3$ with both ends at a boundary in ${\cal E}(2)$ or two arcs on level $3$ with 
			 a boundary in ${\cal E}(2).$
			In the former case, we can extend the interval $I^{(2)}$ to cover the arcs $X$,that is the arc $A^{(3)}$, including the interval $I^{(1)}$. 
			Then, redefining the interval $I^{(1)}$ as $I^{(1)} = W$ and defining $I^{(3)} = Z^\ast,$ 
			we guarantee that the four intervals $I^{(0)}, \dots, I^{(3)}$ define a question of the desired type.
			In the latter case, we can extend the interval $I^{(2)}$ to cover the arcs $A^{(3)}$ and $X$, that is different from $A^{(3)}$. 
			Then, defining $I^{(3)} = Z^\ast,$ it guarantees that the four intervals $I^{(0)}, \dots, I^{(3)}$ define a question of the desired type.

			\item $I^{(3)} \neq \emptyset.$ Then, by the observations above, as a result of the construction on level 2, either there is an 
			arc on level $3$ with both ends at a boundary in ${\cal E}(2)$ or each arc on level $3$ has a boundary in ${\cal E}(2).$ In the latter case, we can clearly extend the intervals $I^{(2)}$ and $I^{(3)}$ in order to cover $d$ elements on level 3.
			In the former case, we can extend the interval $I^{(2)}$ to cover the arcs $X$ and $A^{(3)}$, including the interval with the boundary at one and of $X$, say $I^{(3)}$. Then, redefining the interval $I^{(3)}$ as $I^{(3)} = Z^\ast,$ it guarantees that the four intervals $I^{(0)}, \dots, I^{(3)}$ define a question of the desired type.
		\end{itemize}	
\end{enumerate}
\end{enumerate}
Since in all the cases, the conditions of Lemma \ref{lemma:well-shapeness} are satisfied, the resulting states are well shaped.
\qed
\end{proof}

\begin{lemma}
\label{th:specialized-YES2}
Let $\sigma$ be a well-shaped state with type $\tau (\sigma) = (1,1,c_2,d_2) $, 
for every integer $0 \leq d \leq d_2,$
there exists a $4$-interval  question $Q$ of type $|Q|   = [1,0,2,d]$ such that 
the resulting "yes" and "no" states are \emph{well-shaped}.
\end{lemma}

\begin{proof}
 Let $E$ be a non empty arc on level 2 and 
 let $e^+$ its boundary with $U \in \{P,Q,R\}$. We use one interval $I^{(0)}$ to cover one element from 
 the mode $S$. If $|E| \geq 2 $, we take an interval $I^{(1)}$ on the arc $E$, in particular there exists $e^\ast $ 
 in $E$ such that the sub arc $E^\ast \subseteq E$ with  $|E^\ast| = 2$ is  covered by $I^{(1)}$. 
 In order to cover $d$ elements on level 3, we extend the interval $I^{(1)}$ over $U$ to cover up to $|U|$ elements. 
 Then, we use the remaining two intervals $I^{(2)}, I^{(3)}$ to cover the remaining $d-|U|$ elements over the arcs $\{P,Q,R\}\setminus \{U\}$. 
 In this case, we have that the interval $I^{(0)}$ is a mode-splitting interval. The interval $I^{(1)}$ either is a mode-splitting interval or it is a downward step-splitting interval while a mode at level 2 is completely uncovered. The intervals $I^{(2)}$ and $I^{(3)}$ involves only arcs on level 3, thus they are not involved in the arc covering and spitting on the lower levels. The conditions of Lemma \ref{lemma:well-shapeness} are satisfied, then we have that both the resulting states are well shaped.
 \noindent
 Otherwise, $|E| = 1$, let $F$ be a non empty arc, different from $E$, on level 2 and let $f^+$ its boundary with $V \in \{P,Q,R\} \setminus \{U\}$. Then, we use
 the interval $I^{(1)}$ to cover entirely the arc $E$, and there exists $f^\ast \subseteq F$ such that the sub arc $F^\ast \in F$ is the arc spanning from $f^\ast$ to $f^+$ which has $|F^\ast| = 1$,  covered by an additional interval $I^{(2)}$. In order to cover $d$ elements on level 3, as before, we extend the interval $I^{(1)}$ over $U$ to cover up to $|U|$ elements and we extend the interval $I^{(2)}$ over $V$ to cover up to $|V|$ elements.
Finally, if necessary, the remaining interval $I^{(3)}$ is used  to include the remaining $d -|U| - |V|$ 
elements over the last arc $\{P,Q,R\} \setminus \{U,V\}$.
In this case, we have that the interval $I^{(0)}$ is a mode-splitting interval. The interval $I^{(1)}$ either is a mode-covering interval or it is a step-covering interval while the interval $I^{(2)}$ is either a mode-splitting interval or it is a downward step-splitting interval. If the interval $I^{(2)}$ is a downward step-splitting interval and the interval $I^{(1)}$ is a mode-covering interval, then there exists a completely uncovered mode. The interval $I^{(3)}$ involves only arcs on level 3, thus they are not involved in the arc covering and spitting on the lower levels. The conditions of Lemma \ref{lemma:well-shapeness} are satisfied, then we have that both the resulting states are well shaped. \qed
\end{proof}

\begin{lemma}
\label{th:specialized-lemma5}
Let $\sigma$ be a well-shaped state of type $\tau (\sigma) = (1,0,3,d_3) $. For every integer $0 \leq d \leq d_3$ there 
there exists a $4$-interval question $Q$ of type  $|Q| = [1,0,0,d]$ such that  the resulting "yes" and "no" states are \emph{well-shaped}.
\end{lemma}

\begin{proof}

We use one interval $I^{(0)}$ to cover one element from the mode $A$. Then we use the remaining three intervals $I^{(1)} \dots I^{(3)}$ to cover $d$ elements over the remaining three arcs, respectively $P$, $Q$ and $R$. 
We have that the interval $I^{(0)}$ is a mode-covering interval, while the intervals $I^{(1)}, ..., I^{(3)}$ involves only arcs at level 3 that does not affect the splitting covering of the lower levels. The conditions of Lemma 2 are satisfied and the resulting states are well shaped. \qed
\end{proof}

\subsection{The proof of the main theorem}

\begin{definition}[nice and 4-intervals nice state] \label{def:4-interval-nice}
Following the a standard terminology in the area, we say that a state $\sigma$ is {\em nice} if  there exists a strategy {\cal S} of size $ch(\sigma)$ that is winning 
for $\sigma$. 
In addition we say that a state $\sigma$ is {\em 4-interval nice} if 
\begin{enumerate}
	\item $\sigma$ is well shaped.
	\item there exists a strategy $S$ of size $ch(\sigma)$ that is winning for $\sigma$ and only uses $4$-interval questions
\end{enumerate}
\end{definition}

A direct consequence of the above definition is that a state $\sigma$ is $4$-interval nice if it is well-shaped and, either it is a final state, 
or there is a $4$-interval question for $\sigma$ such that both the resulting $yes$ and $no$ states are $4$-interval nice.

With this definition we can state the following, which gives a sufficient condition for the claim of Theorem \ref{th:non-asymptotic}.
\begin{lemma}
Fix an integer $m \in \mathcal{N}$. If every well-shaped state of type  $(1, m, {m \choose 2}, {m \choose 3})$ is $4$-interval nice then 
there exists a $4$-interval question perfect strategy in the game with $3$ lies over the space of size $2^m$.
\end{lemma}
\begin{proof}
In the light of Definition \ref{def:4-interval-nice} we will show that if $(1, m, {m \choose 2}, {m \choose 3})$ is  $4$-interval nice then 
$(2^m, 0, 0, 0)$ is $4$-interval nice. 

We have the following claim directly following from Theorem \ref{th:generalized-questions}.

\noindent
{\em Claim}
For each $j = 0, 1, \dots m-1$ let $\sigma^{(j)} = (2^{m-j}, j 2^{m-j}, {j \choose 2} 2^{m-j}, {j \choose 3} 2^{m-j}).$ By Theorem Theorem \ref{th:generalized-questions} 
if $\sigma$ is a well-shaped state of type $\sigma^{(j)}$ then there is a $4$-interval question $Q$ of type 
$|Q| = [\frac{|\sigma^{-1}(0)|}{2},\frac{|\sigma^{-1}(1)|}{2},\frac{|\sigma^{-1}(2)|}{2},\frac{|\sigma^{-1}(3)|}{2}]$ such that both the resulting $yes$ and $no$ state
are well shaped of type $\sigma^{(j+1)}$.

Therefore, starting from the well-shaped state $(2^m, 0, 0, 0)$ by repeated application of the Claim, we have that there is a sequence of 
$4$-interval questions such that for every possible sequence of answers the resulting state is 
well-shaped and of type $(1, m, {m \choose 2}, {m \choose 3})$. Since all the above questions are balanced, we also have that 
$ch(1, m, {m \choose 2}, {m \choose 3}) = ch(2^m, 0, 0, 0) - m.$

Therefore if it is true that every well-shaped state of type $(1,m,{m \choose 2},{m \choose 3}) $ is 4-intervals nice, it follows that 
$(2^m, 0, 0, 0)$ is also $4$-interval nice, as desired. \qed
\end{proof}

The following lemma is actually a rephrasing (in the present terminology) of Proposition \ref{proposition:end_game}.

\begin{lemma}\label{corollary:1-mode}
For $n \in \mathbb{N}$ let $T(n) = \{(1,0,0,n), 
(0,1,0,n), (0,0,1,n), (0,0,0,n)\}$. 
Fix $n \in \mathbb{N}$ and let $\sigma$ be a state of type $\tau \in  T(n).$ 
Then $\sigma$ is 4-intervals nice.
\end{lemma}
\begin{proof}
First we observe that every state in $\sigma$ is well-shaped since it has one arc on level $j$ and one arc on level $4$. 
The existence of a perfect strategy only using $4$-interval questions is the a direct consequence of Proposition \ref{proposition:end_game}.
\qed
\end{proof}

\begin{definition}[$0$-typical state \cite{guzicki1990ulam}\cite{negro1992ulam}]

Let $\sigma$ be a state of type $(t_0,t_1,t_2,t_3)$ with $ch(t_0,t_1,t_2,t_3) = q$. We say that 
 $\sigma$ is $0$-typical if  the following hold
$$t_0 = 0; \quad t_2 \geq t_1-1; \quad t_3 \geq q.$$
\end{definition}

\begin{lemma}\label{corollary:Lemma2}
    Let $\sigma$ be a 0-typical well shaped state of character $\geq 12$, then 
    there exists a $4$-interval question $Q$ such that both the resulting states 
    $\sigma_{yes}$ and $\sigma_{no}$ are $0$-typical well-shaped and $ch(\sigma_{yes}) < ch(\sigma)$ and 
    $ch(\sigma_{no}) < ch(\sigma).$
\end{lemma}
\begin{proof}
Let $(0, t_1, t_2, t_3)$ be the type of $\sigma$.
From  \cite[Theorems 3.3, 3.4]{guzicki1990ulam} and \cite[Lemma 2]{negro1992ulam} we have that if the type of $Q$ is defined according to the following cases, then 
both the resulting states $\sigma_{yes}$ and $\sigma_{no}$ are $0$-typical  and in addition $ch(\sigma_{yes}) < ch(\sigma)$ and 
    $ch(\sigma_{no}) < ch(\sigma).$

\begin{itemize}
	\item[Case 1.] $ch(\sigma) = k \geq 3$ and $t_2 \geq 3k-3$
	
	\noindent
	\begin{itemize}
		\item [{\em Case 1.1.}]  $t_1$ and $t_2$ are even numbers.
		
		Then $Q$ is chosen of type $[0, \frac{t_1}{2}, \frac{t_2}{2}, \lfloor \frac{t_3}{2} \rfloor].$
		
		\item [{\em Case 1.2.}]  $t_1$ is even and $t_2$ is odd.
		
		Then $Q$ is chosen of type $[0, \frac{t_1}{2}, \lfloor \frac{t_2}{2} \rfloor, \lfloor \frac{t_3+k-1}{2} \rfloor].$
		
		\item [{\em Case 1.3.}]  $t_1$ is odd and $t_2$ is even.
		
		Then $Q$ is chosen of type\footnote{In \cite[Theorem 3.3]{guzicki1990ulam} the equivalent question of type 
		  $[0, \lfloor \frac{t_1}{2} \rfloor, \frac{t_2}{2} + B_1, \lfloor \frac{t_3- t -1}{2} \rfloor]$ is used.}
		 $[0, t_1-\lfloor \frac{t_1}{2} \rfloor, t_2-(\frac{t_2}{2} + B_1), t_3-\lfloor \frac{t_3- t -1}{2} \rfloor],$
		where $B_1 = \lfloor \frac{\lfloor \frac{k^2-k+2}{2(k-1)} \rfloor}{2} \rfloor$ and $t = (2B_1+1)(k-1) - \frac{k^2-k+2}{2}.$
		
		\item [{\em Case 1.4.}]  $t_1$ and $t_2$ are odd.
		
		Then $Q$ is chosen of type\footnote{In \cite[Theorem 3.3]{guzicki1990ulam} the equivalent question,  type 
		  $[0, \lfloor \frac{t_1}{2} \rfloor, \lfloor \frac{t_2}{2} \rfloor + B_1, \lfloor \frac{t_3- t -1}{2} \rfloor],$ is used.}
		   $[0, t_1-\lfloor \frac{t_1}{2} \rfloor, t_2-(\lfloor \frac{t_2}{2} \rfloor + B_1), t_3-\lfloor \frac{t_3- t -1}{2} \rfloor],$
		where $B_1 = \lceil \frac{\lfloor \frac{k^2-k+2}{2(k-1)} \rfloor}{2} \rceil$ and $t = 2B_1(k-1) - \frac{k^2-k+2}{2}.$
		
	\end{itemize}
	
	\item[Case 2.]	$ch(\sigma) = k \geq 4$ and $t_3 \geq k^2.$ 

	\noindent
	\begin{itemize}
		\item [{\em Case 2.1.}]  $t_1$ and $t_2$ are even numbers.
		
		Then $Q$ is chosen of type  $[0, \frac{t_1}{2}, \frac{t_2}{2}, \lfloor \frac{t_3}{2} \rfloor].$
		
		\item [{\em Case 2.2.}]  $t_1$ is even and $t_2$ is odd.
		
		Then $Q$ is chosen of type  $[0, \frac{t_1}{2}, \lfloor \frac{t_2}{2} \rfloor, \lfloor \frac{t_3+k-1}{2} \rfloor].$
		
		\item [{\em Case 2.3.}]  $t_1$ is odd and $t_2$ is even.
		
		Then $Q$ is chosen of type $[0, \lfloor \frac{t_1}{2} \rfloor, \frac{t_2}{2} , \lfloor \frac{t_3 + t}{2} \rfloor],$
		where $t =  (\frac{k^2-3k+2}{2}) $ .
		
		\item [{\em Case 2.4.}]  $t_1$ and $t_2$ are odd.
		
		Then $Q$ is chosen of type\footnote{In \cite[Theorem 3.4]{guzicki1990ulam} the equivalent question of type 
		  $[0, \lfloor \frac{t_1}{2} \rfloor, \lceil \frac{t_2}{2} \rceil, \lfloor \frac{t_3 + t}{2} \rfloor],$ is used.}
		  $[0, t_1 - \lfloor \frac{t_1}{2} \rfloor, t_2 - \lceil \frac{t_2}{2} \rceil, t_3 - \lfloor \frac{t_3 + t}{2} \rfloor],$
		where $t = ( \frac{k^2-5k+4}{2}) $.
		
	\end{itemize}

\end{itemize}

By Lemma \ref{th:specialized-guzicki}, for each one of the type considered in the above 8 cases, there is a $4$-interval question
of that type such that the resulting states are well-shaped. This completes the proof of the claim.  \qed
\end{proof}

\medskip
\noindent
{\bf The $0$-typical states of character $\leq 13.$}
Let $\tilde{W}$ be the set of quadruples $(0, t_1, t_2, t_3)$ such that for each $\tau \in \tilde{W}$ the following two conditions are satisfied:
\begin{itemize}
\item every state $\sigma$ of type $\tau$ is nice
\item there exist $0\leq \alpha_1 \leq \lceil t_1/2\rceil$ and $0 \leq \alpha_2 \leq \lceil t_2/2 \rceil$ and $0 \leq \alpha_3 \leq t_3$ 
such that asking a question of type $[0, \alpha_1, \alpha_2, \alpha_3]$ in a state $\sigma$ of type $\tau$ both the resulting states
$\sigma_{yes}$ and $\sigma_{no}$ have character smaller than $\sigma$.
\end{itemize}

Note that if $\sigma$ is a well shaped state whose character is in $\tilde{W}$ then it is also $4$-interval nice, by Lemma \ref{th:specialized-guzicki}.

We can exhaustively compute all the states in $\tilde{W}$ of character $\leq 12$. More precisely, we will compute all the pairs $t_1, t_2,$ such that there 
exists $t_3$ and $(0, t_1, t_2, t_3) \in \tilde{W}$; $ch(0, t_1, t_2, t_3) \leq 12.$

Let $$\tilde{M}(t_1, t_2) = \begin{cases}
1 & t_1 = t_2 = 0;\\
0 & t_1 = 1, t_2 = 0 \mbox{ or } t_1 = 0, t_2 = 1;\\
\displaystyle{\min_{\substack{0\leq \alpha_1\leq \lceil \frac{t_1}{2} \rceil, \\ 0\leq \alpha_2\leq \lceil \frac{t_2}{2} \rceil}} \tilde{MinC}(t_1, t_2, \alpha_1, \alpha_2) }
& otherwise,
\end{cases}
$$

where setting $\tilde{\sigma} = (0, t_1, t_2, 0)$ and
$\tilde{\sigma}_{yes} = (0, y_1, y_2, y_3)$ and $\tilde{\sigma}_{no} = (0, n_1, n_2, n_3)$ being the resulting $yes$ and $no$ states when question 
$[0, \alpha_1, \alpha_2]$ is asked in state $\tilde{\sigma}$ and 
$$k_1 = ch(\tilde{\sigma}_{yes}), \, k_2 = ch(\tilde{\sigma}_{no}), \, k_3 = ch(0, y_1, y_2, \tilde{M}(y_1, y_2)), \, 
k_4 = ch(0, n_1, n_2, \tilde{M}(n_1, n_2)),$$
we have $\displaystyle{\tilde{MinC}(t_1, t_2, \alpha_1, \alpha_2) = \min \{t_3 \mid ch(0, t_1, t_2, t_3) > \max\{k_1, k_2, k_3, k_4\}\}.}$

It is not hard to see that the quantities $\tilde{M}$ and $\tilde{MinC}$ can be computed by a dynamic programming approach. The following 
proposition shows that they correctly characterize quadruples in $\tilde{W}$.

\begin{proposition} \label{prop:exhaustive}
The quadruple $\tau = (0, t_1, t_2, d)$ is in $\tilde{W}$ if and only if $d \geq \tilde{M}(t_1, t_2).$ 
\end{proposition}
\begin{proof}
The statement immediately follows by \cite[Definition 3.14, Corollary 3.15, Proposition 3.16]{guzicki1990ulam} 
where it is shown that a state $(0, t_1, t_2, t_3)$ is nice 
if and only if $t_3 \geq M(t_1, t_2)$ where $M(t_1, t_2)$ is defined like $\tilde{M}(t_1, t_2)$ but for the 
fact that $\alpha_i$ can be as large as $t_i$ for $i=1, 2.$ \qed
\end{proof}

 By using the above functions, we can have an algorithm that 
 exhaustively compute all the 0-typical states of character $\leq 13$ which are not in $\tilde{W}.$ It turns out all such states have 
 character $\leq 11.$ The states are reported in Table \ref{tab:not-nice}. 
 It turns out that the output of the computation of $\tilde{M}(t_1,t_2)$ coincides with the output of the computation of $M(t_1, t_2)$, the function 
 considered in \cite{guzicki1990ulam} which does not require $\alpha_i \leq \lceil \frac{t_i}{2}\rceil$ ($i=1, 2$).
 In other words, every well shaped state of character $\leq 13$ which is nice is also $4$-interval nice. 
 
 The following proposition summarizes the above discussion.
\begin{proposition}\label{corollary:Lemma2}
Let $\sigma$ be a 0-typical well shaped state of type $(0,t_1,t_2,t_3).$ If $ch(0,t_1,t_2,t_3) \geq 12$, 
then $\sigma$ is 4-intervals nice. If $ch(0,t_1,t_2,t_3) \leq 12$ then $\sigma$ is 4-intervals nice
 unless it is one of the states in Table \ref{tab:not-nice}.
\end{proposition}

\begin{table}[!htp]
{\small
		\centering
\begin{adjustbox}{max width=\textwidth}
			\begin{tabular}{rllrll}%
				\toprule
				{Character} & {State} & {Values of $t_3$} & {Character} & {State} & {Values of $t_3$} \\
				\otoprule
				$ch = 6$ & $(0,1,5,t_3)$ & $6 \leq t_3 \leq 7$  & $ch = 9$ & $(0,7,16,t_3)$ & $9 \leq t_3 \leq 30$\\
                & $(0,2,1,t_3)$ & $6 \leq t_3 \leq 13$  &  & $(0,7,17,t_3)$ & $9 \leq t_3 \leq 20$\\
				& $(0,2,2,6)$ &   &  & $(0,7,18,t_3)$ & $9 \leq t_3 \leq 10$\\				
				&  &   &  & $(0,8,9,t_3)$ & $9 \leq t_3 \leq 54$\\				
				&  &   &  & $(0,8,10,t_3)$ & $9 \leq t_3 \leq 44$\\				
				$ch = 7$ & $(0,2,7,t_3)$ & $7 \leq t_3 \leq 14$  &  & $(0,8,11,t_3)$ & $9 \leq t_3 \leq 34$\\
				& $(0,3,2,t_3)$ & $7 \leq t_3 \leq 25$  &  & $(0,8,12,t_3)$ & $9 \leq t_3 \leq 24$\\
				& $(0,3,3,t_3)$ & $7 \leq t_3 \leq 17$  &  & $(0,8,13,t_3)$ & $9 \leq t_3 \leq 14$\\		
				& $(0,3,4,t_3)$ & $7 \leq t_3 \leq 9$  &  & $(0,9,8,t_3)$ & $9 \leq t_3 \leq 18$\\		
				&  &   &  &  & \\		
				&  &   & $ch = 10$ & $(0,13,25,t_3)$ & $10 \leq t_3 \leq 21$\\		
				&  &   &  & $(0,13,26,10)$ & \\		
				$ch = 8$ & $(0,3,15,t_3)$ & $8 \leq t_3 \leq 10$  &  & $(0,14,17,t_3)$ & $10 \leq t_3 \leq 53$\\
				& $(0,4,9,t_3)$ & $8 \leq t_3 \leq 27$  &  & $(0,14,18,t_3)$ & $10 \leq t_3 \leq 42$\\
				& $(0,4,10,t_3)$ & $8 \leq t_3 \leq 18$  &  & $(0,14,19,t_3)$ & $10 \leq t_3 \leq 31$\\
				& $(0,4,11,t_3)$ & $8 \leq t_3 \leq 9$  &  & $(0,14,20,t_3)$ & $10 \leq t_3 \leq 20$\\
				& $(0,5,4,t_3)$ & $8 \leq t_3 \leq 35$  &  & $(0,15,14,t_3)$ & $10 \leq t_3 \leq 30$\\
				& $(0,5,5,t_3)$ & $8 \leq t_3 \leq 26$  &  & $(0,15,15,t_3)$ & $10 \leq t_3 \leq 19$\\
				& $(0,5,6,t_3)$ & $8 \leq t_3 \leq 17$  &  &  & \\
				& $(0,5,7,8)$ &   & $ch = 11$ & $(0,25,30,t_3)$ & $11 \leq t_3 \leq 13$\\
				\bottomrule
			\end{tabular}
		\end{adjustbox}
		\caption{The table shows all the states of character $\leq 12$ which are 0-typical but non nice \cite{guzicki1990ulam}\cite{negro1992ulam}. This set coincides with  
			the set of types of well shaped states that are 0-typical but non 4-interval nice\label{tab:not-nice}.}
}
\end{table}

\begin{lemma}
\label{lemma:lemma5}
Let $\sigma$ be a well-shaped state of type $(1,0,3,n)$ with $n \geq 7$, then $\sigma$ is $4$-interval nice. Namely, there exists a $4$-interval question $Q$ such that both the resulting states $\sigma_{yes}$ and $\sigma_{no}$ are $4$-interval nice.
\end{lemma}
\begin{proof}
From \cite[Lemma 5]{negro1992ulam} for $7\leq n \leq 9$, we have that if the question $Q$ is defined as in Table \ref{tab:lemma5} then the resulting states $\sigma_{yes}$ and $\sigma_{no}$ have character strictly smaller than $ch(\sigma)$ and are nice. 
In addition, by Lemmas \ref{th:specialized-guzicki} and  \ref{th:specialized-lemma5}, we have that each question reported in the table can be implemented using at most $4$-intervals.  Moreover, the resulting states are also $4$-interval nice as consequence of Lemma \ref{corollary:1-mode} and as consequence of the observation that the state $(0,0,4,0)$ corresponds to the state $(4,0)$ that is nice by \cite{Pel0} and, since the state has only four elements, it is straightforward to implement every question with at most $4$-intervals. From those observations, it follows that a well shaped state of type $(1,0,3,n)$, for $7 \leq n \leq 9$ is $4$-interval nice.

Let us now consider the case $n \geq 10$. Let the question $Q$ be a question of type $[1,0,3,x]$, then the resulting states $\sigma_{yes}$ and $\sigma_{no}$ have types $(1,0,0,3 + x)$ and $(0,1,3,n-x)$ respectively.
Again from \cite[Lemma 5]{negro1992ulam} for $n \geq 10$, we have that if $x = \lfloor (n+3q+{q \choose 3})/2 \rfloor$, where $ch(\sigma) = q+1$,  the resulting states $\sigma_{yes}$ and $\sigma_{no}$ are nice.
Moreover, from Lemma \ref{th:specialized-lemma5} it follows that the question $Q$ can be implemented using at most $4$ intervals, then we have that the resulting states $\sigma_{yes}$ and $\sigma_{no}$ are also well shaped.
Indeed, note that $\sigma_{yes}$ is $4$-interval nice by Lemma \ref{corollary:1-mode} while $\sigma_{no}$ is $0$-typical, then it is $4$-interval nice by Proposition \ref{corollary:Lemma2}. The proof is complete \qed

\end{proof}
\begin{table}[!htp]
{\small
		\centering
		\caption{First part of the proof of Lemma \ref{lemma:lemma5} and first part of the proof of \cite[Lemma 5]{negro1992ulam}\label{tab:lemma5}}
			\begin{adjustbox}{max width=\textwidth}
			\begin{tabular}{rlllll}%
				\toprule
				{Case} & {State}   		  & {Question} & {Implementation} & Yes-State & No-State \\
				\otoprule
				$n=7$ 	& $(1,0,3,7)$ 	& $[1,0,0,2]?$  & Lemma \ref{th:specialized-lemma5} & $(1,0,0,5)$\textsuperscript{a} & $(0,1,3,5)$\\
				&$(0,1,3,5)$ 		& $[0,1,0,5]?$  & Lemma \ref{th:specialized-guzicki}& $(0,1,0,8)$\textsuperscript{a} & $(0,0,4,0)$\textsuperscript{b}\\
				\cmidrule(lr){1-6}
                $n=8$ 	& $(1,0,3,8)$ 	& $[1,0,0,3]?$  & Theorem \ref{th:specialized-lemma5} & $(1,0,0,6)$\textsuperscript{a} & $(0,1,3,5)$\\
				&$(0,1,3,5)$ 		& $[0,1,0,5]?$  & Lemma \ref{th:specialized-guzicki}& $(0,1,0,8)$\textsuperscript{a} & $(0,0,4,0)$\textsuperscript{b}\\
				\cmidrule(lr){1-6}
				$n=9$ 	& $(1,0,3,9)$ 	& $[1,0,0,4]?$  & Theorem \ref{th:specialized-lemma5} & $(1,0,0,7)$\textsuperscript{a} & $(0,1,3,5)$\\
				&$(0,1,3,5)$ 		& $[0,1,0,5]?$  & Lemma \ref{th:specialized-guzicki}& $(0,1,0,8)$\textsuperscript{a} & $(0,0,4,0)$\textsuperscript{b}\\
				\bottomrule
			\end{tabular}
			\end{adjustbox}
\justify
	\small\textsuperscript{a}This state is 4-intervals nice by Lemma~\ref{corollary:1-mode} \\
	\small\textsuperscript{b}This state is nice by~\cite{Pel0} and since the state has only four elements all the questions can be implemented using at most 4 intervals

}
\end{table}

\begin{lemma}\label{Corollary:Lemma7}
Fix $m \geq 33$ and let $\sigma$ be a well shaped state of type $(1,m,{m \choose 2}, {m \choose 3})$. 
Then,  there exists a sequence of three $4$-interval questions such that all  
the resulting states are 4 interval-nice and of characters $\leq ch(\sigma) - 3.$
\end{lemma}

\begin{proof}
By \cite[Lemma 7]{negro1992ulam} for a state $\sigma$ satisfying the hypothesis of the lemma,  
there exists a sequence of three questions such that all the resulting states are nice and of character $\leq ch(\sigma)-3.$ 
In particular, denoting by $(1,b_0,c_0,d_0)$ the state of type $(1,m,{m \choose 2}, {m \choose 3})$ the questions are defined according to the 
following rules, where $Q_i$ is the $i$th question ($i=1, 2, 3,$) and 
$YES_i$ (resp. $NO_i$) denotes the state resulting from the answer $yes$ (resp.\ $no$) to question $Q_i.$

Let $ q = ch(1,m,{m \choose 2}, {m \choose 3}) - 2$.
	
	Let $Q_1$ be a question of type $|Q_1| = [1,\lfloor \frac{b_0}{2} \rfloor,\lfloor \frac{c_0}{2} \rfloor - \lfloor \frac{b_0}{2} \rfloor , x]?$. With 
	$x=\lfloor \frac{\alpha}{2}\rfloor$ where 
	$\alpha = d_0 + 2\lfloor \frac{c_0}{2} \rfloor - c_0 - 2\lfloor \frac{b_0}{2} \rfloor - {q+1 \choose 3} + {q+1 \choose 2}(b_0 + 1 -2\lfloor \frac{b_0}{2} \rfloor) + (q+2)(c_0 + 4\lfloor \frac{b_0}{2} \rfloor  - b_0 -2\lfloor \frac{c_0}{2} \rfloor ).$

Thus, $YES_1 = (1,\lfloor \frac{b_0}{2}\rfloor, b_0 - 2\lfloor \frac{b_0}{2}\rfloor + \lfloor \frac{c_0}{2}\rfloor , c_0 - \lfloor \frac{c_0}{2}\rfloor + \lfloor \frac{b_0}{2}\rfloor + x)$ and $NO_1 = (0, b_0 + 1 - \lfloor \frac{b_0}{2}\rfloor , 2\lfloor \frac{b_0}{2}\rfloor + c_0 - \lfloor \frac{c_0}{2}\rfloor, \lfloor \frac{c_0}{2}\rfloor - \lfloor \frac{b_0}{2}\rfloor +d_0 - x ) $ are the resulting states.

Let us now denote the type of $YES_1$ as $(1,b_1,c_1,d_1)$ then let $Q_2$ be a question of type $|Q_2| = [1,1,\lfloor \frac{c_1}{2} \rfloor - \lfloor \frac{b_1}{2} \rfloor , y]?$ with $y = \lfloor \frac{\beta}{2} \rfloor$ where 
$\beta = (b_1 -1) {q \choose 2} - {q \choose 3} + (q+1)(c_1 +2 -b_1 + 2\lfloor \frac{b_1}{2} \rfloor   -2\lfloor \frac{c_1}{2} \rfloor )+
d_1 -c_1 + 2\lfloor \frac{c_1}{2} \rfloor -2\lfloor \frac{b_1}{2} \rfloor.$

Thus, the type of the resulting states $YES_2$ and $NO_2$ are given by 
$|YES_2| = (1,1,b_1 -1 +\lfloor \frac{c_1}{2}\rfloor - \lfloor \frac{b_1}{2}\rfloor , c_1 - \lfloor \frac{c_1}{2}\rfloor + \lfloor \frac{b_1}{2}\rfloor +y )$ 
and $|NO_2| = (0,b_1, 1+ c_1 - \lfloor \frac{c_1}{2}\rfloor + \lfloor \frac{b_1}{2}\rfloor, \lfloor \frac{c_1}{2}\rfloor - \lfloor \frac{b_1}{2}\rfloor + d_1 -y ).$ 
	
Let us denote the type of $YES_2$ by $(1,b_1,c_1,d_1)$ then question $Q_3$ is chosen to be of  type $|Q_3| = [1,0,2, z]?$ with
$z = \Bigg\lfloor \frac{ d_2 + 4 -c_2 - {q-1 \choose 3} + 2{q-1 \choose 2} +q(c_2-5)}{2} \Bigg\rfloor.$

We have that the resulting states have type $|YES_3| = (1,0,3,c_2-2+z)$ and $|NO_3| =( 0,2,c_2-2 , 2+d_2-z).$

\smallskip

It remains to show that each one of the above questions can be implemented using only 4-intervals and guaranteeing that the resulting state is also
well-shaped. This is true of $Q_1$ by Theorem \ref{th:generalized-questions}; in addition 
question $Q_2$ applied to the state $YES_1$ satisfies the constraints of Lemma \ref{th:specialized-YES1}. Finally, 
question $Q_3$ asked in state $YES_2$ satisfies the constraints of Lemma \ref{th:specialized-YES2}.

Finally, we observe that the states
$NO_1$, $NO_2$ and $NO_3$ are 4 interval-nice by Proposition \ref{corollary:Lemma2}  and the states 
$YES_3$ is 4 interval-nice by Lemma \ref{lemma:lemma5}.
\qed
\end{proof}

\begin{corollary} \label{corollary:theoremNS}
Fix $m \in \mathcal{N} = \mathbb{N} \setminus \{2,3,5 \}$ and let $\sigma$ be a well shaped state of type 
$(1,m,{m \choose 2},{m \choose 3}).$ 
Then $\sigma$ is 4-intervals nice.
\end{corollary}

\begin{proof}

The searching strategy for $m = 1$ is summarized in Table \ref{tab:m1}. Since the number of elements is $\leq 2$ trivially every question
is $4$-interval. 
The strategy for $m = 4$ are summarized in Table \ref{tab:m4}. The reason for the implementability of each question involved 
by only using $4$-intervals is indicated in the column "Implementation" referring to the appropriate lemma. 

\smallskip

Before proceeding to the analysis of the remaining cases, we note 
that given two states $\sigma = (t_0,t_1,t_2,t_3 )$ and $\sigma' = (t'_0,t'_1,t'_2,t'_3 )$ 
such that $ch(\sigma) = ch(\sigma')$ and for all $i = 0,1,2,3$ it holds that $t'_i \leq t_i$, 
a perfect strategy for $\sigma$ immediately gives a perfect strategy for $\sigma'$.
Therefore, in order to prove the claim  for  $6 \leq m \leq 32$, it suffice to consider the cases 
$m \in \{8,12,17,23,32\}$. In fact for each $6 \leq m \leq 32$ such that $m \not \in \{8, 12, 17, 23, 32\}$ there exists 
$m' \in  \{8,12,17,23,32\}$ such that  $m' \geq m$ and $ch(2^m, 0, 0, 0) = ch(2^{m'}, 0, 0, 0).$ Hence by the above observation 
a perfect $4$-interval strategy for the former implies a perfect $4$-interval strategy for the latter. 

\smallskip

For each $m \in \{8,12,17,23,32\}$ a $4$-interval perfect strategy is described in Table \ref{tab:lemma6}. The 
analysis of the first three levels of the search tree together with the results of Lemma \ref{lemma:lemma5} and
 Proposition \ref{corollary:Lemma2} is sufficient to prove the 4-intervals niceness. 
Finally,   for every $m \geq 33$, the state $(1,m,{m \choose 2}, {m \choose 3})$ is 4-intervals nice by Lemma \ref{Corollary:Lemma7}. \qed
\end{proof}

\begin{table}[!htp]
{\small
	\centering
	\caption{Case $m =1$, after the first question and case $m =4$, after the first four questions\label{tab:m4}\label{tab:m1}}
		\begin{adjustbox}{max width=\textwidth}
			\begin{tabular}{cccccc}%
				\toprule
				&{State}   		  & {Question} & {Implementation}  & Yes-State & No-State \\
				\otoprule
				$m = 1$ &$(1,1,0,0)$ & $[1,0,0,0]?$ & Straightforward  & $(1,0,1,0)$ & $(0,2,0,0)$\\
				&$(1,0,1,0)$ & $[1,0,0,0]?$ & Straightforward   & $(1,0,0,1)$\textsuperscript{a} & $(0,1,1,0)$\\
					&$(0,2,0,0)$ & $[0,1,0,0]?$ & Straightforward & $(0,1,1,0)$ & $(0,1,1,0)$\\
					&$(0,1,1,0)$ & $[0,1,0,0]?$ & Straightforward  & $(0,1,0,1)$\textsuperscript{a} & $(0,0,2,0)$\textsuperscript{b}\\
				\cmidrule(lr){1-6}
				$m = 4$ & $(1,4,6,4)$ & $[1,1,3,2]?$  & Lemma \ref{th:specialized-YES1}  & $(1,1,6,5)$ & $(0,4,4,5)$\\
				&$(1,1,6,5)$ & $[1,0,2,3]?$  & Lemma \ref{th:specialized-YES2}  & $(1,0,3,7)$\textsuperscript{c} & $(0,2,4,4)$\\
				&$(0,4,4,5)$ & $[0,2,2,3]?$  & Lemma \ref{th:specialized-guzicki} & $(0,2,4,5)$ & $(0,2,4,4)$\\
				&$(0,2,4,4)$ & $[0,1,2,2]?$  & Lemma \ref{th:specialized-guzicki} & $(0,1,3,4)$ & $(0,1,3,4)$\\
				&$(0,2,4,5)$ & $[0,1,2,3]?$  & Lemma \ref{th:specialized-guzicki} & $(0,1,3,5)$ & $(0,1,3,4)$\\
				&$(0,1,3,4)$ & $[0,1,0,4]?$  & Lemma \ref{th:specialized-guzicki} & $(0,1,0,7)$\textsuperscript{a} & $(0,0,4,0)$\textsuperscript{b}\\
				&$(0,1,3,5)$ & $[0,1,0,5]?$  & Lemma \ref{th:specialized-guzicki} & $(0,1,0,8)$\textsuperscript{a} & $(0,0,4,0)$\textsuperscript{b}\\
				
				\bottomrule
			\end{tabular}
		\end{adjustbox}
	\justify
	\small\textsuperscript{a} This state is 4-intervals nice by Lemma~\ref{corollary:1-mode} \\
	\small\textsuperscript{b}This state is nice by~\cite{Pel0} and since the state has less then four elements all the questions can be implemented using at most 4 intervals
}
\end{table}

\begin{table}[!htp]
		\centering
		\caption{Proof of Corollary \ref{corollary:theoremNS} for $6 \leq m \leq 32$ and proof of \cite[Lemma 6]{negro1992ulam}\label{tab:lemma6}}
		\begin{adjustbox}{max width=\textwidth}
			\begin{tabular}{rlllll}%
				\toprule
				{Case} & {State}   		  & {Question} & {Implementation} & Yes-State & No-State \\
				\otoprule
				$m=8$ 	& $(1,8,28,56)$ 	& $[1,4,10,22]?$  & Theorem \ref{th:generalized-questions} & $(1,4,14,40)$ & $(0,5,22,44)$\textsuperscript{a}\\
				&$(1,4,14,40)$ 		& $[1,1,5,36]?$  & Lemma \ref{th:specialized-YES1}& $(1,1,8,45)$ & $(0,4,10,9)$\textsuperscript{a}\\
				&$(1,1,8,45)$		& $[1,0,2,29]?$  & Lemma \ref{th:specialized-YES2}& $(1,0,3,35)$\textsuperscript{b} & $(0,2,6,18)$\textsuperscript{a}\\
				\cmidrule(lr){1-6}
				$m=12$ 	& $(1,12,66,220)$ 	& $[1,6,27,110]?$  & Theorem \ref{th:generalized-questions} & $(1,6,33,149)$ & $(0,7,45,137)$\textsuperscript{a}\\
				&$(1,6,33,149)$ 	& $[1,1,13,136]?$  & Lemma \ref{th:specialized-YES1} & $(1,1,18,156)$ & $(0,6,21,26)$\textsuperscript{a}\\
				&$(1,1,18,156)$ 	& $[1,0,2,120]?$  & Lemma \ref{th:specialized-YES2} & $(1,0,3,136)$\textsuperscript{b} & $(0,2,16,38)$\textsuperscript{a}\\
				\cmidrule(lr){1-6}
				$m=17$ 	& $(1,17,136,680)$ 	& $[1,8,60,373]?$  & Theorem \ref{th:generalized-questions} & $(1,8,69,449)$ & $(0,10,84,367)$\textsuperscript{a}\\
				&$(1,8,69,449)$ 	& $[1,1,30,344]?$  & Lemma \ref{th:specialized-YES1} & $(1,1,37,383)$ & $(0,2,35,69)$\textsuperscript{a}\\
				&$(1,1,37,383)$ 	& $[1,0,2,316]?$  & Lemma \ref{th:specialized-YES2} & $(1,0,3,351)$\textsuperscript{b} & $(0,2,35,69)$\textsuperscript{a}\\
				\cmidrule(lr){1-6}
				$m=23$ 	& $(1,23,253,1771)$ & $[1,11,115,946]?$  & Theorem \ref{th:generalized-questions} & $(1,11,127,1084)$ & $(0,13,149,940)$\textsuperscript{a}\\
				&$(1,11,127,1084)$ 	& $[1,1,58,767]?$  & Lemma \ref{th:specialized-YES1} & $(1,1,68,836)$ & $(0,11,70,375)$\textsuperscript{a}\\
				&$(1,1,68,836)$ 	& $[1,0,2,700]?$  & Lemma \ref{th:specialized-YES2} & $(1,0,3,766)$\textsuperscript{b} & $(0,2,66,138)$\textsuperscript{a}\\
				\cmidrule(lr){1-6}
				$m=32$ 	&$(1,32,496,4960)$ 	& $[1,16,232,2545]?$  & Theorem \ref{th:generalized-questions} & $(1,16,248,2809)$ & $(0,17,280,2647)$\textsuperscript{a}\\
				&$(1,16,248,2809)$ 	& $[1,1,116,1852]?$  & Lemma \ref{th:specialized-YES1} & $(1,1,131,1984)$ & $(0,16,133,1073)$\textsuperscript{a}\\
				&$(1,1,131,1984)$ 	& $[1,0,2,1635]?$  &Lemma \ref{th:specialized-YES2} & $(1,0,3,1764)$\textsuperscript{b} & $(0,2,129,351)$\textsuperscript{a}\\
				
				\bottomrule
			\end{tabular}
		\end{adjustbox}
	\justify
	\small\textsuperscript{a} This state is 4-intervals nice by Proposition~\ref{corollary:Lemma2} \\
	\small\textsuperscript{b} This state is 4-intervals nice by Lemma~\ref{lemma:lemma5} 
\end{table}

\bibliographystyle{elsarticle-num}

\end{document}